\documentclass[12pt]{article}
\usepackage{amsmath}
\usepackage{latexsym}
\usepackage{amsfonts}

\def\Bbb{\mathbb}

\def\eea{\end{eqnarray*}}

\newtheorem{main}{Theorem}

\newtheorem{defn}{Definition}
\newtheorem{thm}{Theorem}
\newtheorem{prop}[thm]{Proposition}
\newtheorem{cor}[thm]{Corollary}
\newtheorem{lem}[thm]{Lemma}

\newtheorem{Pro}[thm]{Problem}

\newenvironment{proof}{\medskip \noindent
{\bf Proof.}}{\hfill \rule{.5em}{1em}
\\}

\newenvironment{rmk}{\mbox{ }\\{\bf  Remark}\mbox{ }}{
\hfill $\Box$\mbox{}\bigskip}

\begin{document}

\title{The Normalized Ricci Flow on Four-Manifolds and Exotic Smooth Structures}

\author{Masashi Ishida} 

\date{}

\maketitle

\begin{abstract}
In this article, we shall investigate the relationship between the existence or non-existence of non-singular solutions to the normalized Ricci flow and smooth structures on closed 4-manifolds, where non-singular solutions to the normalized Ricci flow are solutions which exist for all time $t \in [0, \infty)$ with uniformly bounded sectional curvature. In dimension 4, there exist many compact topological manifolds admitting distinct smooth structures, i.e., exotic smooth structures. Interestingly, in this article, the difference between existence and non-existence of non-singular solutions to the normalized Ricci flow on 4-manifolds turns out to strictly depend on the choice of smooth structure. In fact, we shall prove that, for every natural number $\ell$, there exists a compact topological 4-manifold $X_{\ell}$ which admits smooth structures for which non-singular solutions of the normalized Ricci flow exist, but also admits smooth structures for which {\it no} non-singular solution of the normalized Ricci flow exists. Hence, in dimension 4, smooth structures become definite obstructions to the existence of non-singular solutions to the normalized Ricci flow. 
\end{abstract}


\section{Introduction}\label{intro}

Let $X$ be a closed oriented Riemannian manifold of dimension $n \geq 3$. {The Ricci flow} on $X$ is the following evolution equation:
\begin{eqnarray*}
 \frac{\partial }{\partial t}{g}=-2{Ric}_{g}, 
\end{eqnarray*}
where ${Ric}_{g}$ is the Ricci curvature of the evolving Riemannian metric $g$. The Ricci flow was firstly introduced in the celebrated work \cite{ha-0} of Hamilton for producing the constant positive sectional curvature metrics on 3-manifolds. Since the above equation does not preserve volume in general, one often considers the normalized Ricci flow on $X$:
\begin{eqnarray*}\label{Ricci}
 \frac{\partial }{\partial t}{g}=-2{Ric}_{g} + \frac{2}{n}\overline{s}_{g} {g}, 
\end{eqnarray*}
where $\overline{s}_{g}:={{\int}_{X} {s}_{g} d{\mu}_{g}}/{vol_{{g}}}$ and ${s}_{g}$ denotes the scalar curvature of the evolving Riemannian metric $g$, $vol_{g}:={\int}_{X}d{\mu}_{g}$ and $d{\mu}_{g}$ is the volume measure with respect to $g$. A one-parameter family of metric $\{g(t)\}$, where $t \in [0, T)$ for some $0<T\leq \infty$, is called a solution to the normalized Ricci flow if this satisfies the above equation at all $x \in X$ and $t \in [0, T)$. It is known that the normalized flow is equivalent to the unnormalized flow by reparametrizing in time $t$ and scaling the metric in space by a function of $t$. The volume of the solution metric to the normalized Ricci flow is constant in time. \par
The key point of an approach for understanding the topology of a given manifold via the normalized Ricci flow is to get the long-time behavior of the solution. Recall that a solution $\{g(t)\}$ to the normalized Ricci flow on a time interval $[0, T)$ is said to be maximal if it cannot be extended past time $T$. Let us also recall the following definition firstly introduced by Hamilton \cite{ha-1, c-c}:
\begin{defn}\label{non-sin}
A maximal solution $\{g(t)\}$, $t \in [0, T)$, to the normalized Ricci flow on $X$ is called non-singular if $T=\infty$ and the Riemannian curvature tensor $Rm_{g(t)}$ of $g(t)$ satisfies 
$$
\sup_{X \times [0, T)}|Rm_{g(t)}| < \infty. 
$$
\end{defn} 
As a pioneer work, Hamilton \cite{ha-0} proved that, in dimesion 3, there exists a unique non-singular solution to the normalized Ricci flow if the initial metric is positive Ricci curvature. Moreover, Hamilton \cite{ha-1} classified non-singular solutions to the normalized Ricci flow on 3-manifolds and the work was very important for understanding long-time behaivor of solutions of the Ricci flow on 3-manifolds. On the other hand, many authors studied the properties of non-singular solutions in higer dimensions. For example, Hamilton \cite{ha-2} proved that, for any closed oriented Riemannian 4-manifold with constant positive curvature operator, there is a unique non-singular solution to the normalized flow which converges to a smooth Riemannian metric of positive sectional curvature. On the other hand, it is known that the solution on a 4-manifold with positive isotropic curvature definitely becomes singular \cite{ha-3, ha-4}. See also a recent very nice work of Chen and Zhu \cite{c-z} on Ricci flow with surgery on 4-manifolds with positive isotropic curvature inspired by the work of Hamilton \cite{ha-3} and the celebrated work of Perelman \cite{p-1, p-2, p-3, c-c, lott, m-t}. There is also an interesting work concerning Ricci flow on homogeneous 4-manifolds due to Isenberg, Jackson and Lu \cite{isen}. See also Lott's work \cite{lott-r} concerning with the long-time behavior of Type-III Ricci flow solutions on homogeneous manifolds. However, the existence and non-existence of non-singular solutions to the normalized Ricci flow in higher dimensions $n \geq 4$ are still mysterious in general. The main purpose of this article is to study, from the gauge theoretic point of view, this problem in case of dimension 4 and point out that the difference between existence and non-existence of non-singular solutions to the normalized Ricci flow strictly depend on one's choice of smooth structure. The main result of the present article is Theorem \ref{main-A} stated below. \par
In \cite{fz-1}, Fang, Zhang and Zhang also studied the properties of non-singular solutions to the normalized Ricci flow in higher dimensions. Inspired by their work, we shall introduce the following definition:
\begin{defn}\label{bs}
A maximal solution $\{g(t)\}$, $t \in [0, T)$, to the normalized Ricci flow  on $X$ is called quasi-non-singular if $T=\infty$ and the scalar curvature $s_{g(t)}$ of $g(t)$ satisfies 
$$
\sup_{X \times [0, T)}|{s}_{g(t)}| < \infty. 
$$
\end{defn}
Of course, the condition of Definition \ref{bs} is weaker than that of Definition \ref{non-sin}. Namely, any non-singular solution is quasi-non-singular, but the converse is not true in general. In dimension 4, the authors of \cite{fz-1} observed, among others, that any closed oriented smooth 4-manifold $X$ must satisfy the following {\it topological} constraint on the Euler characteritic $\chi(X)$ and signature $\tau(X)$ of $X$:
\begin{eqnarray}\label{FZZ}
2 \chi(X) \geq 3|\tau(X)|
\end{eqnarray}
if there is a quasi-non-singular solution to the normalized Ricci flow on $X$ and, moreover, if the solution satisfies 
\begin{eqnarray}\label{FZZ-s}
\hat{s}_{g(t)} \leq -c <0, 
\end{eqnarray}
where the constant $c$ is independent of $t$ and define as $\hat{s}_{g} := \min_{x \in X}{s}_{g}(x)$ for a given Riemannian metric $g$. In this article, we shall call the inequality (\ref{FZZ}) the {\it Fang-Zhang-Zhang inequality} (or, for brevity, FZZ inequality) for the normalized Ricci flow and we shall also call $2 \chi(X) > 3|\tau(X)|$ the {\it strict} FZZ inequality for the normalized Ricci flow. The FZZ inequality gives us, under the condition (\ref{FZZ-s}), the only known topological obstruction to the existence of quasi-non-singular solutions to the normalized Ricci flow. It is also known that any Einstein 4-manifold $X$ must satisfy the same bound $2 \chi(X) \geq 3|\tau(X)|$ which is so called Hitchin-Thorpe inequality \cite{thor, hit}. We notice that, however, under the bound (\ref{FZZ-s}), (quasi-)non-singular solutions do not necessarily converge to smooth Einstein metrics on $X$. Hence, FZZ inequality never follows from Hitchin-Thorpe inequality in general. See \cite{fz-1} for more details. \par
On the other hand, there is a natural diffeomorphism invariant arising from a variational problem for the total scalar curvature of Riemannian metrics on any given closed oriented Riemannian manifold $X$ of dimension $n\geq 3$. As was conjectured by Yamabe \cite{yam}, and later proved by Trudinger, Aubin, and Schoen \cite{aubyam,lp,rick,trud}, every conformal class on any smooth compact manifold contains a Riemannian metric of constant scalar curvature. For each conformal class $[g]=\{ vg ~|~v: X\to {\Bbb R}^+\}$, we are able to consider an associated number $Y_{[g]}$ which is so called {\em Yamabe constant} of the conformal class $[g]$ defined by 
\begin{eqnarray*}
Y_{[g]} = \inf_{h \in [g]}  \frac{\int_X 
s_{{h}}~d\mu_{{h}}}{\left(\int_X 
d\mu_{{h}}\right)^{\frac{n-2}{n}}}, 
\end{eqnarray*}
where $d\mu_{{h}}$ is the volume form with respect to the metric $h$. The Trudinger-Aubin-Schoen theorem tells us that this number is actually realized as the constant scalar curvature of some unit volume metric in the conformal class $[g]$. Then, Kobayashi \cite{kob} and Schoen \cite{sch} independently introduced the following invariant of $X$:
\begin{eqnarray*}
{\mathcal Y}(X) = \sup_{\mathcal{C}}Y_{[g]}, 
\end{eqnarray*}
where $\mathcal{C}$ is the set of all conformal classes on $X$. This is now commonly known as the {\em Yamabe invariant} of $X$. It is known that ${\mathcal Y}(X) \leq 0$ if and only if $X$ does not admit a metric of positive scalar curvature. There is now a substantial literature \cite{ishi-leb-2,leb-4, leb-7, leb-11,jp2, jp3, petyun} concerning manifolds of non-positive Yamabe invariant, and the exact value of the invariant is computed for a large number of these manifolds. In particular, it is also known that the Yamabe invariant is sensitive to the choice of smooth structure of a 4-manifold. After the celebrated works of Donaldson \cite{don, don-kro} and Freedman \cite{free}, it now turns out that quite many exotic smooth structures exist in dimension 4. Indeed, there exists a compact topological 4-manifold $X$ which admits many distinct smooth structures ${Z}^i$. Equivalently, each of the smooth 4-manifolds ${Z}^i$ is homeomorphic to $X$, but never diffeomorphic to each!
  other. One can construct quite many explicite examples of compact topological 4-manifolds admitting distinct smooth structures for which values of the Yamabe invariants are different by using, for instance, a result of LeBrun with the present author \cite{ishi-leb-2}. \par
Now, let us come back to the Ricci flow picture. In this article, we shall observe that the condition (\ref{FZZ-s}) above is closely related to the negativity of the Yamabe invariant of a given smooth Riemannian manifold. More precisely, in Proposition \ref{yama-b} proved in Section \ref{ya} below, we shall see that the condition (\ref{FZZ-s}) is always satisfied for {\it any} solution to the normalized Ricci flow if a given smooth Riemannian manifold $X$ of dimension $n \geq 3$ has ${\mathcal Y}(X)<0$. Moreover, we shall also observe that, in Theorem \ref{bound-four} in Section \ref{ya} below, if a compact topological 4-manifold $M$ admits a smooth structure $Z$ with ${\mathcal Y}<0$ and for which there exists a non-singular solution to the normalized Ricci flow, then the strict FZZ inequality for $Z$ must hold:
$$
2 \chi(Z) > 3|\tau(Z)|, 
$$
where, of course, we identified the compact topological 4-manifold $M$ admits the smooth structure $Z$ with the smooth 4-manifold $Z$. Let us here emphasize that $2 \chi(Z) > 3|\tau(Z)|$ is just a topological constraint, is {\it not} a differential topological one. The observations made in this article and the special feature of smooth structures in dimension 4 naturally lead us to ask the following:
\begin{Pro}\label{Q}
Let $X$ be any compact topological 4-manifold which admits at least two distinct smooth structures $Z^i$ with negative Yamabe invariant ${\mathcal Y}<0$. Suppose that, for at least one of these smooth structures $Z^i$, there exist non-singular solutions to the the normalized Ricci flow. Then, for every other smooth structure $Z^i$ with ${\mathcal Y}<0$, are there always non-singular solutions to the normalized Ricci flow?
\end{Pro}

Since $X$ admits, for at least one of these smooth structures $Z^i$, non-singular solutions to the the normalized Ricci flow, we are able to conclude that $2 \chi({Z}_{i}) > 3|\tau({Z}_{i})|$ holds for every $i$. Notice that this is equivalent to $2 \chi(X) > 3|\tau(X)|$. Hence, even if there are always non-singular solutions to the normalized Ricci flow for every other smooth structure $Z^i$, it dose not contradict the strict FZZ inequality. \par
Interestingly, the main result of this article tells us that the answer to Problem \ref{Q} is negative as follows:
\begin{main}\label{main-A}
For every natural number $\ell$, there exists a simply connected compact topological non-spin 4-manifold $X_{\ell}$ satisfying the following properties:
\begin{itemize}
\item $X_{\ell}$ admits at least $\ell$ different smooth structures $M^i_{\ell}$ with ${\mathcal Y}<0$ and for which there exist non-singular solutions to the the normalized Ricci flow in the sense of Definition \ref{non-sin}. Moreover the existence of the solutions forces the strict FZZ inequality $2 \chi > 3|\tau|$ as a topological constraint, 
\item $X_{\ell}$ also admits infinitely many different smooth structures  $N^j_{\ell}$ with ${\mathcal Y}<0$ and for which there exists no quasi-non-singular solution to the normalized Ricci flow in the sense of Definition \ref{bs}. In particular, there exists no non-singular solution to the the normalized Ricci flow in the sense of Definition \ref{non-sin}. 
\end{itemize} 
\end{main} 
Notice that Freedman's classification \cite{free} implies that $X_{\ell}$ above must be homeomorphic to a connected sum $p{\mathbb C}{P}^2 \# q \overline{{\mathbb C}{P}^2}$, where ${\mathbb C}{P}^2$ is the complex projective plane and ${\mathbb C}{P}^2$ is the complex projective plane with the reversed orientation, and $p$ and $q$ are some appropriate positive integers which depend on the natural number $\ell$. Notice also that, for the standard smooth structure on $p{\mathbb C}{P}^2 \# q \overline{{\mathbb C}{P}^2}$, we have ${\mathcal Y} >0$ because, by a result of Schoen and Yau \cite{rick-yau} or Gromov and Lawson \cite{g-l}, there exists a Riemannian metric of positive scalar curvature for such a smooth structure. Hence, smooth structures which appear in Theorem \ref{main-A} are far from the standard smooth structure. On the other hand, notice also that the second statement of Theorem \ref{main-A} tells us that the topological 4-manifold $X_{\ell}$ admits infinitely many different smooth structures  $N^j_{\ell}$ with ${\mathcal Y}<0$ and for which any solution to the normalized Ricci flow always becomes singular for any initial metric. In the case of 4-manifolds with ${\mathcal Y} >0$, for example, consider a smooth 4-manifold with positive isotropic curvature metric $g$ and with no essential incompressible space form. Then it is known that the Ricci flow develops singularites for the initial metric $g$. The structure of singularites is studied deeply by Hamilton \cite{ha-3, ha-4} and Chen and Zhu \cite{c-z}. In the present article, however, we do not pursue this issue in our case ${\mathcal Y}<0$. \par 
To the best of our knowledge, Theorem \ref{main-A} is the first result which shows that, in dimension 4, smooth structures become definite obstructions to the existence of non-singular solutions to the normalized Ricci flow. Namely, Theorem \ref{main-A} teaches us that the existence or non-existence of non-singular solutions to the normalized Ricci flow depends strictly on the diffeotype of a 4-manifold and it is {\it not} determined by homeotype alone. This gives a completely new insight into the property of solutions to the Ricci flow on 4-manifolds. \par
To prove the non-existence result in Theorem \ref{main-A}, we need to prove new obstructions to the existence of non-singular solutions to the normalized Ricci flow. Indeed, it is the main non-trivial step in the proof of Theorem \ref{main-A}. For instance, we shall prove the following obstruction:
\begin{main}\label{main-B}
Let $X$ be a closed symplectic 4-manifold with $b^{+}(X) \geq 2$ and $2 \chi(X) + 3\tau(X)>0$, where ${b}^{+}(X)$ stands for the dimension of a maximal positive definite subspace of ${H}^{2}(X, {\mathbb R})$ with respect to the intersection form. Then, there is no non-singular solution of the normalized Ricci flow on a connected sum $M:=X \# k{\overline{{\mathbb C}{P}^2}}$ if 
\begin{eqnarray*}
k \geq \frac{1}{3}\Big(2 \chi(X) + 3\tau(X) \Big). 
\end{eqnarray*}
\end{main}
See also Theorem \ref{ricci-ob-1} and Theorem \ref{ricci-ob-2} below for more general obstructions. We shall use the Seiberg-Witten monopole equations \cite{w} to prove the obstructions. We should notice that, under the same condition, LeBrun \cite{leb-11} firstly proved that $M$ above cannot admit any Einstein metric by using Seiberg-Witten monopole equations. As was already mentioned above, notice that, however, (quasi-)non-singular solutions do not necessarily converge to smooth Einstein metrics on $M$ under the bound (\ref{FZZ-s}). Hence, the above non-existence result on non-singular solutions never follows from the obstruction of LeBrun in general. In this sense, the above obstruction in Theorem \ref{main-B} is new and non-trivial. On the other hand, to prove the existence result of non-singular solutions in Theorem \ref{main-A}, we shall use a very nice result of Cao \cite{c, c-c} concerning the existence of non-singular solutions to the normalized Ricci flow on compact K{\"{a}}hler manifolds. By combining non-existence result derived from Theorem \ref{main-B} with the existence result of Cao, we shall give a proof of Theorem \ref{main-A}. \par
The organization of this article is as follows. In Section \ref{ht}, we shall recall the proof of the FZZ inequality (\ref{FZZ}) because we shall use, in Section \ref{obstruction} below, the idea of the proof to prove new obstructions to the existence of non-singular solutions to the normalized Ricci flow. In Section 3, first of all, we shall prove that the condition (\ref{FZZ-s}) above is always satisfied for any solution to the normalized Ricci flow if a given Riemannian manifold $X$ has negative Yamabe invariant. Moreover, we shall improve the FZZ inequality (\ref{FZZ}) under an assumption that a given Riemannian manifold $X$ has negative Yamabe invariant. This motivates Problem \ref{Q} partially. See Theorem \ref{bound-four} below. In Section \ref{monopoles}, we shall discuss curvature bounds arising from the Seiberg-Witten monopole equations. In fact, we shall firstly recall, for the reader who is unfamiliar with Seiberg-Witten theory, these curvature bounds following a!
  recent beautiful article \cite{leb-17} of LeBrun. And we shall prove, by using the curvature bounds, some results which are needed to prove the new obstructions. The main results of this section are Theorems \ref{yamabe-pere} and \ref{key-mono-b} below. In Section \ref{obstruction}, we shall prove the new obstructions by gathering results proved in the previous several sections. See Theorem \ref{ricci-ob-1}, Corollary \ref{non-sin-cor} (Theorem \ref{main-B}) and Theorem \ref{ricci-ob-2} below. In Section \ref{final-main}, we shall finally give a proof of the main theorem, i.e., Theorem \ref{main-A}, by using particularly Corollary \ref{non-sin-cor} (Theorem \ref{main-B}). Finally, in Section \ref{remark}, we shall conclude this article by giving some open questions which are closely related to Theorem \ref{main-A}. \par
The main part of this work was done during the present author's stay at State University of New York at Stony Brook in 2006. I would like to express my deep gratitude to Claude LeBrun for his warm encouragements and hospitality. I would like to thank the Department of Mathemathics of SUNY at Stony Brook for their hospitality and nice atmosphere during the preparation of this article.

\section{Hitchin-Thorpe Type Inequality for the Normalized Ricci Flow}\label{ht}

In this section, we shall recall the proof of the Fang-Zhang-Zhang inequality (\ref{FZZ}) for the normalized Ricci flow. We shall use the idea of the proof, in Section 5 below, to prove new obstructions. We notice that, throughout the article \cite{fz-1}, the authors of \cite{fz-1} assume that any solution $\{g(t)\}$, $t \in [0, \infty)$, to the normalized Ricci follow has {\it unite volume}, namely, ${vol}_{g(t)}=1$ holds for all $t \in [0, \infty)$. Since the normalized Ricci flow preserves the volume of the solution, this condition is equivalent to the condition that ${vol}_{g(0)}=1$ for the initial metric $g(0)$. Though one can always assume this condition by rescaling the metic, such a condition is not essential. In what follows, let us give a proof of the FZZ inequality without such a condition on the volume. Lemma \ref{FZZ-lem}, Proposition \ref{FZZ-prop} and Theorem \ref{fz-key} below are essentially due to the authors of \cite{fz-1}. We shall include its proof for completeness and the reader's convenience. \par
Now, let $X$ be a closed oriented Riemannian 4-manifold. Then, the Chern-Gauss-Bonnet formula and the Hirzebruch signature formula tell us that the following formulas hold for {\it any} Riemannian metric $g$ on $X$:
\begin{eqnarray*}
\tau(X)=\frac{1}{12{\pi}^2}{\int}_{X}\Big(|W^+_{g}|^2-|W^-_{g}|^2 \Big) d{\mu}_{g}, \\
\chi(X) = \frac{1}{8{\pi}^2}{\int}_{X}\Big(\frac{{s}^2_{g}}{24}+|W^+_{g}|^2+|W^-_{g}|^2-\frac{|\stackrel{\circ}{r}_{g}|^2}{2}  \Big) d{\mu}_{g}, 
\end{eqnarray*}
where $W^+_{g}$ and $W^-_{g}$ denote respectively the self-dual and anti-self-dual Weyl curvature of the metric $g$ and $\stackrel{\circ}{r}_{g}$ is the trace-free part of the Ricci curvature of the metric $g$. And $s_{g}$ is again the scalar curvature of the metric $g$ and $d\mu_{{g}}$ is the volume form with respect to $g$. By these formulas, we are able to get the following important equality:
\begin{eqnarray}\label{4-im}
2\chi(X) \pm 3\tau(X) = \frac{1}{4{\pi}^2}{\int}_{X}\Big(2|W^{\pm}_{g}|^2+\frac{{s}^2_{g}}{24}-\frac{|\stackrel{\circ}{r}_{g}|^2}{2} \Big) d{\mu}_{g}.  
\end{eqnarray}
If $X$ admits an Einstein metric $g$, then we have $\stackrel{\circ}{r}_{g} \equiv 0$. The above formula therefore implies that any Einstein 4-manifold must satisfies
$$
2 \chi(X) \geq 3|\tau(X)|.
$$ 
This is nothing but the Hitchin-Thorpe inequality \cite{thor, hit}. As was already mentioned in Introduction, it is proved that, in \cite{fz-1}, the same inequality still holds for some 4-manifold which is {\it not} necessarily Einstein. Namely, under the existence of quasi-non-singular solutions satisfying the uniform bound (\ref{FZZ-s}) to the normalized Ricci flow, the same inequality still holds. \par
A key observation is the following lemma. This is proved in Lemma 2.7 of \cite{fz-1} for unit volume solution. We would like to point out that the following lemma was already proved essentially by Lemma 7.1 in the article \cite{ha-1} of Hamilton. Notice that, we do not assume that $vol_{{g(t)}}=1$ holds for any $t \in [0, \infty)$:
\begin{lem}\label{FZZ-lem}
Let $X$ be a closed oriented Riemannian manifold of dimension $n$ and assume that there is a quasi-non-singular solution $\{g(t)\}$, $t \in [0, \infty)$, to the normalized Ricci flow in the sense of Definition \ref{bs}. Assume moreover that the solution satisfies the uniform bound (\ref{FZZ-s}), namely, 
\begin{eqnarray*}
\hat{s}_{g(t)} \leq -c <0
\end{eqnarray*}
holds, where the constant $c$ is independent of $t$ and define as $\hat{s}_{g} := \min_{x \in X}{s}_{g}(x)$ for a given Riemannian metric $g$. Then the following two bounds
\begin{eqnarray}\label{fzz-key-0}
{\int}^{\infty}_{0}\Big(\overline{s}_{g(t)}- \hat{{s}}_{g(t)} \Big)dt < \infty,  
\end{eqnarray}
\begin{eqnarray}\label{fzz-key}
{\int}^{\infty}_{0}{\int}_{X}| {{s}}_{g(t)}-\overline{s}_{g(t)} |d{\mu}_{g(t)}dt \leq 2{vol}_{g(0)} {\int}^{\infty}_{0}\Big(\overline{s}_{g(t)}- \hat{{s}}_{g(t)} \Big) dt <\infty
\end{eqnarray}
hold, where $\overline{s}_{g(t)}:={{\int}_{X} {s}_{g(t)} d{\mu}_{g(t)}}/{vol_{{g(t)}}}$. 
\end{lem}
\begin{proof} 
As was already used in Lemma 2.7 in \cite{fz-1}, we shall also use an idea due to Hamilton \cite{ha-1}. More precisely, we shall use the idea of the proof of Lemma 7.1 in \cite{ha-1}. Recall the evolution equation for $s_{g(t)}$: 
\begin{eqnarray*}
\frac{\partial s_{g(t)}}{\partial t} = \Delta s_{g(t)} + 2|Ric_{g(t)}|^2 - \frac{2}{n}\overline{s}_{g(t)}{s}_{g(t)}  
\end{eqnarray*}
which was firstly derived by Hamilton \cite{ha-0}. If we decompose the Ricci tensor $Ric$ into its trace-free part $\stackrel{\circ}{r}$ and its trace $s$, then we have
\begin{eqnarray*}
|Ric_{g(t)}|^2 = |\stackrel{\circ}{r}_{g(t)}|^2 + \frac{1}{n}s_{g(t)}\Big(s_{g(t)}-\overline{s}_{g(t)} \Big) + \frac{\overline{s}_{g(t)}}{n}{s}_{g(t)}. 
\end{eqnarray*}
We therefore obtain the following 
\begin{eqnarray}\label{evolution}
\frac{\partial s_{g(t)}}{\partial t} = \Delta s_{g(t)} + 2|\stackrel{\circ}{r}_{g(t)}|^2 + \frac{2}{n}s_{g(t)}\Big(s_{g(t)}-\overline{s}_{g(t)} \Big). 
\end{eqnarray}
From this, we are able to get the ordinary differential inequality:
\begin{eqnarray*}
\frac{d}{d t}\hat{s}_{g(t)} \geq \frac{2}{n}\hat{s}_{g(t)}\Big(\hat{s}_{g(t)}-\overline{s}_{g(t)} \Big). 
\end{eqnarray*}
Since the solution satisfies the uniform bound (\ref{FZZ-s}), we have
\begin{eqnarray*}
\frac{d}{d t}\hat{s}_{g(t)} \geq \frac{2c}{n}\Big(\overline{s}_{g(t)}-\hat{s}_{g(t)} \Big). 
\end{eqnarray*}
It is clear that this inequality indeed implies the desired bound (\ref{fzz-key-0}). \par
On the other hand, we have the following inequality (see also the proof of Lemma 7.1 in \cite{ha-1}):
\begin{eqnarray*}
\Big|{s}_{g(t)}-\overline{s}_{g(t)}\Big| = \Big|\Big({s}_{g(t)}-\hat{{s}}_{g(t)}\Big)-\Big(\overline{s}_{g(t)}-\hat{{s}}_{g(t)}\Big)\Big| \leq \Big({s}_{g(t)}-\hat{{s}}_{g(t)}\Big)+\Big(\overline{s}_{g(t)}-\hat{{s}}_{g(t)}\Big). 
\end{eqnarray*}
This implies the following:
\begin{eqnarray*}
{\int}_{X}|{s}_{g(t)}-\overline{s}_{g(t)}|d{\mu}_{g(t)} \leq {\int}_{X}\Big({s}_{g(t)}-\hat{{s}}_{g(t)}\Big)d{\mu}_{g(t)} + {\int}_{X}\Big(\overline{s}_{g(t)}-\hat{{s}}_{g(t)}\Big)d{\mu}_{g(t)}. 
\end{eqnarray*}
On the other hand, notice that the following holds: 
\begin{eqnarray*}
{\int}_{X}\overline{s}_{g(t)}d{\mu}_{g(t)}={\int}_{X}\Big(\frac{{\int}_{X} {s}_{g(t)} d{\mu}_{g(t)}}{vol_{{g(t)}}} \Big)d{\mu}_{g(t)} ={\int}_{X}{s}_{g(t)}d{\mu}_{g(t)}. 
\end{eqnarray*}
We therefore obtain 
\begin{eqnarray*}
{\int}_{X}|{s}_{g(t)}-\overline{s}_{g(t)}|d{\mu}_{g(t)} \leq 2{\int}_{X}\Big(\overline{s}_{g(t)}-\hat{{s}}_{g(t)}\Big)d{\mu}_{g(t)}=2{vol}_{g(t)}\Big(\overline{s}_{g(t)}-\hat{{s}}_{g(t)}\Big). 
\end{eqnarray*}
Moreover, as was already mentioned, the normalized Ricci flow preserves the volume of the solution. We therefore have $vol_{g(t)}=vol_{g(0)}$. Hence, 
\begin{eqnarray*}
{\int}_{X}|{s}_{g(t)}-\overline{s}_{g(t)}|d{\mu}_{g(t)} \leq 2{vol}_{g(0)}\Big(\overline{s}_{g(t)}-\hat{{s}}_{g(t)}\Big). 
\end{eqnarray*}
This tells us that 
\begin{eqnarray*}
{\int}^{\infty}_{0}{\int}_{X}| {{s}}_{g(t)}-\overline{s}_{g(t)} |d{\mu}_{g(t)}dt \leq 2{vol}_{g(0)} {\int}^{\infty}_{0}\Big(\overline{s}_{g(t)}- \hat{{s}}_{g(t)}\Big)dt. 
\end{eqnarray*}
This inequality with the bound (\ref{fzz-key-0}) implies the desired bound (\ref{fzz-key}). 
\end{proof} 

Using the above lemma, we are able to show a real key proposition to prove the FZZ inequality. The following result is pointed out in Lemma 3.1 of \cite{fz-1} for unit volume solution and $n=4$
\begin{prop}\label{FZZ-prop}
Let $X$ be a closed oriented Riemannian manifold of dimension $n$ and assume that there is a quasi-non-singular solution $\{g(t)\}$, $t \in [0, \infty)$, to the normalized Ricci flow in the sense of Definition \ref{bs}. Assume moreover that the solution satisfies the uniform bound (\ref{FZZ-s}), namely, 
\begin{eqnarray*}
\hat{s}_{g(t)} \leq -c <0
\end{eqnarray*}
holds, where the constant $c$ is independent of $t$ and define as $\hat{s}_{g} := \min_{x \in X}{s}_{g}(x)$ for a given Riemannian metric $g$. Then, the trace-free part $\stackrel{\circ}{r}_{g(t)}$ of the Ricci curvature satisfies 
\begin{eqnarray}\label{fzz-ricci}
{\int}^{\infty}_{0} {\int}_{X} |\stackrel{\circ}{r}_{g(t)}|^2 d{\mu}_{g(t)}dt < \infty. 
\end{eqnarray}
\end{prop}
\begin{proof}
Now suppose that there exists a quasi-non-singular solution to the normalized Ricci flow on a closed oriented manifold $X$ of dimension $n$. As before, let us consider the evolution equation (\ref{evolution}) for the scalar curvature of the solution:
\begin{eqnarray*}
\frac{\partial s_{g(t)}}{\partial t} = \Delta s_{g(t)} + 2|\stackrel{\circ}{r}_{g(t)}|^2 + \frac{2}{n}s_{g(t)}\Big(s_{g(t)}-\overline{s}_{g(t)} \Big). 
\end{eqnarray*}
Notice that, by the assumption that the solution is quasi-non-singular in the sense of Definition \ref{bs}, we are able to conclude that there is a constant $C$ which is independent of both $t \in [0, \infty)$ and $x \in X$, and $|{s}_{g(t)}|< C$ holds. We therefore obtain
\begin{eqnarray*}
{\int}^{\infty}_{0} {\int}_{X} |\stackrel{\circ}{r}_{g(t)}|^2 d{\mu}_{g(t)}dt &=& \frac{1}{2}{\int}^{\infty}_{0} {\int}_{X} \frac{\partial s_{g(t)}}{\partial t} d{\mu}_{g(t)}dt-\frac{1}{n}{\int}^{\infty}_{0} {\int}_{X} {s_{g(t)}}\Big(s_{g(t)}-\overline{s}_{g(t)} \Big) d{\mu}_{g(t)}dt \\
 &\leq& \frac{1}{2}{\int}^{\infty}_{0} {\int}_{X} \frac{\partial s_{g(t)}}{\partial t} d{\mu}_{g(t)}dt + \frac{1}{n} {\int}^{\infty}_{0} {\int}_{X}| s_{g(t)}| |s_{g(t)}-\overline{s}_{g(t)}| d{\mu}_{g(t)}dt \\
 &=& \frac{1}{2}{\int}^{\infty}_{0} \frac{\partial}{\partial t}\Big({\int}_{X} s_{g(t)} d{\mu}_{g(t)} \Big) dt + \frac{1}{n} {\int}^{\infty}_{0} {\int}_{X}| s_{g(t)}||s_{g(t)}-\overline{s}_{g(t)}| d{\mu}_{g(t)}dt \\
  &\leq& \frac{1}{2}{\int}^{\infty}_{0} \frac{\partial}{\partial t}\Big( \overline{s}_{g(t)}{vol}_{g(t)} \Big) dt + \frac{C}{n} {\int}^{\infty}_{0} {\int}_{X}|s_{g(t)}-\overline{s}_{g(t)}| d{\mu}_{g(t)}dt \\
   &=& \frac{{vol}_{g(0)}}{2}{\int}^{\infty}_{0} \frac{\partial}{\partial t}\overline{s}_{g(t)}dt + \frac{C}{n} {\int}^{\infty}_{0} {\int}_{X}|s_{g(t)}-\overline{s}_{g(t)}| d{\mu}_{g(t)}dt, 
\end{eqnarray*}
where we used a fact that $vol_{g(t)}=vol_{g(0)}$ holds for any $t \in [0, \infty)$. Hence we have 
\begin{eqnarray*}
{\int}^{\infty}_{0} {\int}_{X} |\stackrel{\circ}{r}_{g(t)}|^2 d{\mu}_{g(t)}dt &\leq& \frac{{vol}_{g(0)}}{2}{\int}^{\infty}_{0} \frac{\partial}{\partial t} \overline{s}_{g(t)} dt + \frac{C}{n} {\int}^{\infty}_{0} {\int}_{X}|s_{g(t)}-\overline{s}_{g(t)}| d{\mu}_{g(t)}dt \\
&\leq& \frac{{vol}_{g(0)}}{2}\lim_{t \rightarrow \infty} \sup|\overline{s}_{g(t)}-\overline{s}_{g(0)}| + \frac{C}{n} {\int}^{\infty}_{0} {\int}_{X}|s_{g(t)}-\overline{s}_{g(t)}| d{\mu}_{g(t)}dt. 
\end{eqnarray*}
On the other hand, the uniform bound $|{s}_{g(t)}|< C$ implies
\begin{eqnarray*}
|\overline{s}_{g(t)}|=|\frac{{\int}_{X} {s}_{g(t)} d{\mu}_{g(t)}}{vol_{{g(t)}}}|\leq \frac{{\int}_{X} |{s}_{g(t)}| d{\mu}_{g(t)}}{vol_{{g(t)}}} \leq \frac{{\int}_{X} {C} d{\mu}_{g(t)}}{vol_{{g(t)}}} = C. 
\end{eqnarray*}
This tells us that 
\begin{eqnarray*}
|\overline{s}_{g(t)}-\overline{s}_{g(0)}| \leq |\overline{s}_{g(t)}|+|\overline{s}_{g(0)}| \leq C+C = 2C. 
\end{eqnarray*}
Therefore, we are able to conclude that the following holds:
\begin{eqnarray*}
\sup|\overline{s}_{g(t)}-\overline{s}_{g(0)}| \leq 2C.
\end{eqnarray*}
Hence we obtain 
\begin{eqnarray*}
{\int}^{\infty}_{0} {\int}_{X} |\stackrel{\circ}{r}_{g(t)}|^2 d{\mu}_{g(t)}dt &\leq& \frac{{vol}_{g(0)}}{2} \cdot 2C + \frac{C}{n} {\int}^{\infty}_{0} {\int}_{X}|s_{g(t)}-\overline{s}_{g(t)}| d{\mu}_{g(t)}dt \\
&\leq& {vol}_{g(0)}C+ \frac{C}{n} {\int}^{\infty}_{0} {\int}_{X}|s_{g(t)}-\overline{s}_{g(t)}| d{\mu}_{g(t)}dt. 
\end{eqnarray*}
This estimate with the bound (\ref{fzz-key}) implies
\begin{eqnarray*}
{\int}^{\infty}_{0} {\int}_{X} |\stackrel{\circ}{r}_{g(t)}|^2 d{\mu}_{g(t)}dt < \infty 
\end{eqnarray*}
as promised. 
\end{proof}

As was already noticed in \cite{fz-1}, the bound (\ref{fzz-ricci}) tells us that, when $m \rightarrow \infty$, 
\begin{eqnarray}\label{fzz-ricci-0}
{\int}^{m+1}_{m} {\int}_{X} |\stackrel{\circ}{r}_{g(t)}|^2 d{\mu}_{g(t)}dt \longrightarrow 0 
\end{eqnarray}
 holds since ${\int}_{X} |\stackrel{\circ}{r}_{g(t)}|^2 d{\mu}_{g(t)} \geq 0$. Indeed, one can see this by completely elementary reasons. Particularly, in dimension $n=4$, this (\ref{fzz-ricci-0}) immediately implies the Fang-Zhang-Zhang inequality as follows (See also Lemma 3.2 in \cite{fz-1}):
\begin{thm}\label{fz-key}
Let $X$ be a closed oriented Riemannian 4-manifold and assume that there is a quasi-non-singular solution to the normalized Ricci flow in the sense of Definition \ref{bs}. Assume moreover that the solution satisfies the uniform bound (\ref{FZZ-s}), namely, 
\begin{eqnarray*}
\hat{s}_{g(t)} \leq -c <0
\end{eqnarray*}
holds, where the constant $c$ is independent of $t$ and define as $\hat{s}_{g} := \min_{x \in X}{s}_{g}(x)$ for a given Riemannian metric $g$. Then, $X$ must satisfy
\begin{eqnarray*}
2 \chi(X) \geq 3|\tau(X)|. 
\end{eqnarray*}
\end{thm}

\begin{proof}
Suppose that there exists a quasi-non-singular solution $\{g(t) \}$, $t \in [0, \infty)$, to the normalized Ricci flow on $X$. Assume also that the bound (\ref{FZZ-s}) is satisfied. By the equality (\ref{4-im}) which holds for any Riemannian metric on $X$, we are able to get 
\begin{eqnarray*}
2\chi(X) \pm 3\tau(X) = \frac{1}{4{\pi}^2}{\int}_{X}\Big(2|W^{\pm}_{g(t)}|^2+\frac{{s}^2_{g(t)}}{24}-\frac{|{r}^{\circ}_{g(t)}|^2}{2} \Big) d{\mu}_{g(t)}.  
\end{eqnarray*}
From this and (\ref{fzz-ricci-0}), we are able to obtain 
\begin{eqnarray*}
2\chi(X) \pm 3\tau(X) &=& {\int}^{m+1}_{m} \Big(2\chi(X) \pm 3\tau(X) \Big)dt \\
&=& \frac{1}{4{\pi}^2}{\int}^{m+1}_{m} {\int}_{X}\Big(2|W^{\pm}_{g(t)}|^2+\frac{{s}^2_{g(t)}}{24}-\frac{|{r}^{\circ}_{g(t)}|^2}{2} \Big) d{\mu}_{g(t)}dt \\
&\geq & \liminf_{m \longrightarrow \infty}\frac{1}{4{\pi}^2}{\int}^{m+1}_{m} {\int}_{X}\Big(2|W^{\pm}_{g(t)}|^2+\frac{{s}^2_{g(t)}}{24}-\frac{|{r}^{\circ}_{g(t)}|^2}{2} \Big) d{\mu}_{g(t)}dt \\
&=& \liminf_{m \longrightarrow \infty}\frac{1}{4{\pi}^2}{\int}^{m+1}_{m} {\int}_{X}\Big(2|W^{\pm}_{g(t)}|^2+\frac{{s}^2_{g(t)}}{24}\Big) d{\mu}_{g(t)}dt \geq 0.  
\end{eqnarray*}
We therefore get the desired inequality. 
\end{proof}

\section{Fang-Zhang-Zhang Inequality and Negativity of the Yamabe Invariant}\label{ya}

In this section, we shall improve the FZZ inequality under an assumption that the Yamabe invariant of a given 4-manifold is negative. This motivates partially Problem \ref{Q} which was already mentioned in Introduction. The main result of this section is Theorem \ref{bound-four} below. \par 
Suppose now that $X$ is a closed oriented Riemannian manifold of dimension $n\geq 3$, and moreover that $[g]=\{ ug ~|~u: X \to {\Bbb R}^+\}$ is the conformal class of an arbitrary metric $g$. Trudinger, Aubin, and Schoen \cite{aubyam,lp,rick,trud} proved every conformal class on any smooth compact manifold contains a Riemannian metric of constant scalar curvature. Such a metric $\hat{g}$ can be constructed by minimizing the Einstein-Hilbert functional:
$$
\hat{g}\mapsto  \frac{\int_X 
s_{\hat{g}}~d\mu_{\hat{g}}}{\left(\int_X 
d\mu_{\hat{g}}\right)^{\frac{n-2}{n}}},
$$
among all metrics conformal to $g$. Notice that, by setting $\hat{g} = u^{4/(n-2)}g$, we have the following identity:
$$
\frac{\int_X 
s_{\hat{g}}~d\mu_{\hat{g}}}{\left(\int_X 
d\mu_{\hat{g}}\right)^{\frac{n-2}{n}}}= 
\frac{\int_X\left[ s_gu^2 +
4 \frac{n-1}{n-2}|\nabla u|^2\right] d\mu_g}{\left(\int_X  u^{2n/(n-2)}d\mu_g\right)^{(n-2)/n}}. 
$$
As was already mentioned in Introduction, associated to each conformal class $[g]$, we are able to define the following  number which is called Yamabe constant of the conformal class $[g]$:
$$
Y_{[g]} = \inf_{\hat{g} \in [g]}\frac{\int_X s_{\hat{g}}~d\mu_{\hat{g}}}{\left(\int_X 
d\mu_{\hat{g}}\right)^{\frac{n-2}{n}}}. 
$$
Equivalently, 
$$
Y_{[g]} = \inf_{u \in {C}^{\infty}_{+}(X)}\frac{\int_X\left[ s_gu^2 +
4 \frac{n-1}{n-2}|\nabla u|^2\right] d\mu_g}{\left(\int_X  u^{2n/(n-2)}d\mu_g\right)^{(n-2)/n}}, 
$$
where ${C}^{\infty}_{+}(X)$ is the set of all positive functions $u: X \to {\Bbb R}^+$. Kobayashi \cite{kob} and Schoen \cite{sch} independently introduced the following interesting invariant which is now called Yamabe invariant of $X$:
\begin{eqnarray}\label{yama-def-1}
{\mathcal Y}(X) = \sup_{[g] \in \mathcal{C}} Y_{[g]}, 
\end{eqnarray}
where $\mathcal{C}$ is the set of all conformal classes on $X$. This is a diffeomorphism invariant of $X$. Notice again that ${\mathcal Y}(X) \leq 0$ if and only if $X$ does not admit Riemannian metrics of positive scalar curvature. In this case, it is also known that the Yamabe invariant of $X$ can be rewritten as 
\begin{eqnarray}\label{yama-def-2}
{\mathcal Y}(X) = - \Big(\inf_{g}{\int}_{X}|s_{g}|^{{n}/{2}} d{\mu}_{g} \Big)^{{2}/{n}}, 
\end{eqnarray}
where supremum is taken over all smooth metrics $g$ on $X$. For instance, see Proposition 12 in \cite{ishi-leb-2}. In dimension 4, it is known that there are quite many manifolds whose Yamabe invariants are strictly negative \cite{leb-4, ishi-leb-2}. \par
For any Riemannian metric $g$, consider the minimum $\hat{s}_{g}:=\min_{x \in X}{s}_{g}(x)$ of the scalar curvature ${s}_{g}$ of the metric $g$ as before. In Theorem 2.1 in \cite{ha-1}, Hamilton pointed out that the minimum $\hat{s}_{g}$ is increasing along the normalized Ricci flow when it is non-positive. Hence, it may be interesting to give an upper bound to the quantity. We shall give the following upper bound in terms of the Yamabe invariant. This result is simple, but important for our purpose:
\begin{prop}\label{yama-b}
Let $X$ be a closed oriented Riemannian manifold of dimension $n \geq 3$ and assume that the Yamabe invariant of $X$ is negative, i.e., ${\mathcal Y}(X)<0$. If there is a solution $\{g(t)\}$, $t \in [0, T)$, to the normalized Ricci flow, then the solution satisfies the bound (\ref{FZZ-s}). More precisely, the following is satisfied:
\begin{eqnarray*}
\hat{s}_{g(t)}:=\min_{x \in X}{s}_{g(t)}(x) \leq \frac{{\mathcal Y}(X)}{(vol_{g(0)})^{2/n}} < 0. 
\end{eqnarray*}
\end{prop}
\begin{proof}
Suppose that there is a solution $\{g(t)\}$, $t \in [0, T)$ to the normalized Ricci flow. Let us consider the Yamabe constant $Y_{[g(t)]}$ of a conformal class $[g(t)]$ of a metric $g(t)$ for any $t \in [0, T)$. By definition, we have
$$
{\mathcal Y}(X) \geq Y_{[g(t)]} = \inf_{u \in {C}^{\infty}_{+}(X)}\frac{\int_X\left[ s_{g(t)}u^2 +
4 \frac{n-1}{n-2}|\nabla u|^2\right] d\mu_{g(t)}}{\left(\int_X  u^{2n/(n-2)}d\mu_{g(t)}\right)^{(n-2)/n}}.
$$
We therefore obtain
\begin{eqnarray*}
{\mathcal Y}(X) &\geq& \inf_{u \in {C}^{\infty}_{+}(X)}\frac{{\int}_{X} \Big(\displaystyle\min_{x \in X}{s}_{g(t)}u^2 + 4\frac{n-1}{n-2}|\nabla u|^2 \Big) d{\mu}_{g(t)}}{\Big({\int}_{X}u^{{2n}/{(n-2)}} d{\mu}_{g(t)}\Big)^{{n-2}/{n}}} \\
&\geq& \hat{s}_{g(t)} \Big(\inf_{u \in {C}^{\infty}_{+}(X)}\frac{{\int}_{X} u^2 d{\mu}_{g(t)}}{\Big({\int}_{X}u^{{2n}/{(n-2)}} d{\mu}_{g(t)}\Big)^{{n-2}/{n}}} \Big).
\end{eqnarray*}
where notice that $\hat{s}_{g} := \min_{x \in X}{s}_{g}(x)$. If $\hat{s}_{g(t)} \geq 0$ holds, then the above estimate tells us that ${\mathcal Y}(X) \geq 0$. Since we assume that ${\mathcal Y}(X) < 0$, we are able to conclude that $\hat{s}_{g(t)} < 0$ must hold. \par
On the other hand, the H{\"{o}}lder inequality tells us that the following inequality holds:
\begin{eqnarray*}
{\int}_{X} u^2 d{\mu}_{g(t)} &\leq& \Big({\int}_{X} u^{2n/n-2} d{\mu}_{g(t)} \Big)^{{n-2}/{n}}\Big({\int}_{X} d{\mu}_{g(t)}  \Big)^{2/n} \\
&=& \Big({\int}_{X} u^{2n/n-2} d{\mu}_{g(t)} \Big)^{{n-2}/{n}}{(vol_{g(t)})^{2/n}}. 
\end{eqnarray*}
This implies that
\begin{eqnarray*}
\inf_{u \in {C}^{\infty}_{+}(X)}\frac{{\int}_{X} u^2 d{\mu}_{g(t)}}{\Big({\int}_{X} u^{2n/n-2} d{\mu}_{g(t)} \Big)^{{n-2}/{n}}} \leq {(vol_{g(t)})^{2/n}}.
\end{eqnarray*}
Since we have $\hat{s}_{g(t)} < 0$, this also implies
\begin{eqnarray*}
\hat{s}_{g(t)}\Big( \inf_{u \in {C}^{\infty}_{+}(X)}\frac{{\int}_{X} u^2 d{\mu}_{g(t)}}{\Big({\int}_{X} u^{2n/n-2} d{\mu}_{g(t)} \Big)^{{n-2}/{n}}} \Big) \geq \hat{s}_{g(t)}{(vol_{g(t)})^{2/n}}. 
\end{eqnarray*}
We therefore obtain
\begin{eqnarray*}
{\mathcal Y}(X) &\geq&  \hat{s}_{g(t)}\Big( \inf_{u \in {C}^{\infty}_{+}(X)}\frac{{\int}_{X} u^2 d{\mu}_{g(t)}}{\Big({\int}_{X} u^{2n/n-2} d{\mu}_{g(t)} \Big)^{{n-2}/{n}}} \Big) \\
&\geq& \hat{s}_{g(t)}{(vol_{g(t)})^{2/n}}. 
\end{eqnarray*}
On the other hand, notice that the normalized Ricci flow preserves the volume of the solution. We therefore have $vol_{g(t)}=vol_{g(0)}$ for ant $t \in [0, T)$. Hence, we get the desired bound for any $t \in [0, T)$:
\begin{eqnarray*}
\hat{s}_{g(t)} \leq \frac{{\mathcal Y}(X)}{(vol_{g(t)})^{2/n}}=\frac{{\mathcal Y}(X)}{(vol_{g(0)})^{2/n}} < 0. 
\end{eqnarray*}
In particular, the solution $\{g(t)\}$ satisfies the bound (\ref{FZZ-s}) by setting $-c={{\mathcal Y}(X)}/{(vol_{g(0)})^{2/n}}$. 
\end{proof} 

The following theorem is the main result of this section. Let inj$(x, g)$ be the injectivity radius of the metric $g$ at $x \in X$. Recall that, following \cite{c-g, fz-1}, a solution $\{g(t)\}$ to the normalized Ricci flow on a Riemannian manifold $X$ is called {\it collapse} if there is a sequence of times $t_{k} \rightarrow T$ such that $\sup_{x \in X}$inj$(x, g(t_{k})) \rightarrow 0$, where $T$ is the maximal existence time for the solution, which may be finite or infinite:  
\begin{thm}\label{bound-four}
Let $X$ be a closed oriented Riemannian 4-manifold. Suppose that there is a quasi-non-singular solution $\{g(t)\}$, $t \in [0, \infty)$, to the normalized Ricci flow in the sense of Definition \ref{bs}. If the Yamabe invariant of $X$ is negative, i.e., ${\mathcal Y}(X)<0$, then the following holds:
\begin{eqnarray*}
2 \chi(X) -3|\tau(X)| \geq \frac{1}{96{\pi}^2}|{\mathcal Y}(X)|^2. 
\end{eqnarray*}
In particular, $X$ must satisfy the strict FZZ inequality
\begin{eqnarray*}
2 \chi(X) > 3|\tau(X)| 
\end{eqnarray*}
in this case. Moreover, if the solution is non-singular in the sense of Definition \ref{non-sin}, then the solution does not collapse. 
\end{thm}
\begin{proof}
Suppose that there is a quasi-non-singular solution $\{g(t)\}$, $t \in [0, \infty)$, to the normalized Ricci flow. By Proposition \ref{yama-b} and the assumption that ${\mathcal Y}(X)<0$, the solution automatically satisfies
\begin{eqnarray*}
\hat{s}_{g(t)} \leq \frac{{\mathcal Y}(X)}{(vol_{g(0)})^{1/2}} < 0. 
\end{eqnarray*}
In particular, the solution satisfies the bound (\ref{FZZ-s}). By the proof of Theorem \ref{fz-key} above, we are able to obtain the following bound because there is a quasi-non-singular solution with the uniform bound (\ref{FZZ-s}):
\begin{eqnarray*}
2\chi(X) \pm 3\tau(X) &\geq&  \liminf_{m \longrightarrow \infty}\frac{1}{4{\pi}^2}{\int}^{m+1}_{m} {\int}_{X}\Big(2|W^{\pm}_{g(t)}|^2+\frac{{s}^2_{g(t)}}{24}\Big) d{\mu}_{g(t)}dt \\
&\geq&\ \liminf_{m \longrightarrow \infty}\frac{1}{96{\pi}^2}{\int}^{m+1}_{m} {\int}_{X}{{s}^2_{g(t)}} d{\mu}_{g(t)}dt.  
\end{eqnarray*}
On the other hand, we have the equality (\ref{yama-def-2}) under ${\mathcal Y}(X)<0$. In case where $n=4$, this tells us that
\begin{eqnarray*}
|{\mathcal Y}(X)|^2 = \inf_{g}{\int}_{X}s^{2}_{g} d{\mu}_{g}. 
\end{eqnarray*}
We therefore have
\begin{eqnarray*}
2\chi(X) \pm 3\tau(X) &\geq&  \liminf_{m \longrightarrow \infty}\frac{1}{96{\pi}^2}{\int}^{m+1}_{m} {\int}_{X}{{s}^2_{g(t)}} d{\mu}_{g(t)}dt \\
 &\geq& \frac{1}{96{\pi}^2}|{\mathcal Y}(X)|^2.
\end{eqnarray*}
Since $|{\mathcal Y}(X)| \not=0$, we particularly obtain $2 \chi(X) > 3|\tau(X)|$ as desired. On the other hand, as was already used in \cite{ha-1} and \cite{fz-1}, Cheeger-Gromov's collapsing theorem \cite{c-g} tells us that $X$ must satisfy $\chi(X) = 0$ if it collapses with bounded sectional curvature. However, $X$ now satisfies $2 \chi(X) > 3|\tau(X)|$ and hence $\chi(X) \not=0$. Therefore, we are able to conclude that if the solution is non-singular in the sense of Definition \ref{non-sin}, then the solution does not collapse.
\end{proof} 

\begin{rmk}
It is a natural question to ask whether or not a similar bound holds in the case where ${\mathcal Y}(X) \geq 0$. Suppose now that a given closed 4-manifold $X$ has ${\mathcal Y}(X) > 0$. Notice that the positivity of the Yamabe invariant of $X$ implies the existence of a Riemannian metric $g$ of positive scalar curvature on $X$. According to Proposition 2.2 in \cite{fz-1}, any non-singular solution to the normalized Ricci flow on $X$ with the positive scalar curvature metric $g$ as an initial metric always converges along a subsequence of times to shrinking Ricci soliton $h$. If the $h$ is a {\it gradient} shrinking Ricci soliton, the following bound (\ref{soli}) is known. In this case, there are smooth function $f$ and positive constant $\lambda >0$ satisfying
\begin{eqnarray*}
{Ric}_{h}={\lambda}h+{D}^2f, 
\end{eqnarray*}
where ${D}^2f$ is the Hessian of the Ricci potential function $f$ with respect to $h$. Under the following constraint on the Ricci potential function $f$
\begin{eqnarray*}
{\int}_{X}f{d}{\mu}_{h}=0,  
\end{eqnarray*}
one can see that the following bound holds from the proof of the main theorem of Ma \cite{li-ma}:
\begin{eqnarray}\label{soli}
2\chi(X)-{3}|\tau(X)| \geq \frac{1}{48{\pi}^2}{\mathcal A}_{\frac{3}{2}}(X, h). 
\end{eqnarray}
Here, for any positive constant $a$, define as
\begin{eqnarray*}
{\mathcal A}_{a}(X, h):=a \Big(\frac{({\int}_{X}s_{h} d{\mu}_{h})^2}{vol_{h}}\Big)-{\int}_{X}{s}_{h}^2d{\mu}_{h}. 
\end{eqnarray*}
Notice that, by the Schwarz inequality, ${\mathcal A}_{1}(X, h) \leq 0$ holds. ${\mathcal A}_{\frac{3}{2}}(X, h) \geq 0$ is equivalent to ${\int}_{X}s^2_{h} d{\mu}_{h} \leq 24{\lambda}^2{vol_{h}}$. See the bound (1) in the main theorem of Ma \cite{li-ma}. We also notice that there is a conjecture of Hamilton which asserts that any compact gradient shrinking Ricci soliton with positive curvature operator must be Einstein. See an interesting article of Cao \cite{cao-X} including a partial affirmative answer under a certain integral inequality concerning the Ricci soliton. 
\end{rmk}

On the other hand, let us next recall the definition of Pelerman's $\bar{\lambda}$ invariant \cite{p-1, p-2, lott} briefly. We shall firstly recall an entropy functional which is so called ${\cal F}$-functional introduced and investigated by Perelman \cite{p-1}. Let $X$ be a closed oriented Riemannian manifold of dimension $n$ and $g$ any Riemannian metric on $X$. We shall denote the space of all Riemannian metrics on $X$ by ${\cal R}_{X}$ and the space of all $C^{\infty}$ functions on $X$ by $C^{\infty}(X)$. Then ${\cal F}$-functional is the following functional ${\cal F} : {\cal R}_{X} \times C^{\infty}(X) \rightarrow {\mathbb R}$ defined by
\begin{eqnarray*}
{\cal F}(g, f):={\int}_{X}({s}_{g} + |{\nabla}f|^{2}){e}^{-f} d\mu_{g}, 
\end{eqnarray*}
where $f \in C^{\infty}(X)$, ${s}_{g}$ is again the scalar curvature and $d\mu_{g}$ is the volume measure with respect to $g$. It is then known that, for a given metric $g$, there exists a unique minimizer of the ${\cal F}$-functional under the constraint ${\int}_{X}{e}^{-f} d\mu_{g} =1$. Hence it is so natural to consider the following which is so called Perelman $\lambda$-functional: 
\begin{eqnarray*}
{{\lambda}}_g:=\inf_{f} \ \{ {\cal F}(g, f) \ | \ {\int}_{X}{e}^{-f} d\mu_{g} =1 \}. 
\end{eqnarray*}
It turns out that $\lambda_g$ is nothing but the least eigenvalue of the elliptic operator $4 \Delta_g+s_g$, where $\Delta = d^*d= - \nabla\cdot\nabla $ is the positive-spectrum  Laplace-Beltrami operator associated with $g$. Consider the scale-invariant quantity $\lambda_g (vol_g)^{2/n}$. Then Perelman's $\bar{\lambda}$ invariant of $X$ is defined to be 
\begin{eqnarray*}\label{p-inv}
\bar{\lambda}(X)= \sup_g \lambda_g (vol_g)^{2/n}, 
\end{eqnarray*}
where supremum is taken over all smooth metrics $g$ on $X$. This quantity is closely related to the Yamabe invariant. In fact, the following result holds:
\begin{thm}[\cite{A-ishi-leb-3}]\label{AIL}
Let $X$ be a closed oriented Riemannian $n$-manifold, $n\geq 3$. Then 
\begin{eqnarray*}
\bar{\lambda}(X) = \begin{cases}
     {\mathcal Y}(X) & \text{ if  } {\mathcal Y}(X) \leq 0, \\
     +\infty  & \text{ if  } {\mathcal Y}(X) >  0.
\end{cases}
\end{eqnarray*}
\end{thm}

Theorem \ref{bound-four} and Theorem \ref{AIL} immediately imply 
\begin{cor}\label{bound-four-perel}
Let $X$ be a closed oriented Riemannian 4-manifold. Suppose that there is a quasi-non-singular solution $\{g(t)\}$, $t \in [0, \infty)$, to the normalized Ricci flow in the sense of Definition \ref{bs}. If the Perelman's $\bar{\lambda}$ invariant of $X$ is negative, i.e., $\bar{\lambda}(X)<0$, then the following holds:
\begin{eqnarray*}
2 \chi(X) -3|\tau(X)| \geq \frac{1}{96{\pi}^2}|\bar{\lambda}(X)|^2. 
\end{eqnarray*}
In particular, $X$ must satisfy the strict FZZ inequality $2 \chi(X) > 3|\tau(X)|$ in this case. Moreover, if the solution is non-singular in the sense of Definition \ref{non-sin}, then the solution does not collapse. 
\end{cor}

Notice that this corollary was firstly proved in \cite{fz-1} under the assumption that the solution to the normalized Ricci flow has unit volume. See Theorem 1.4 in \cite{fz-1}.

\section{Curvature Bounds and Convex Hull of the Set of Monopole Classes}\label{monopoles}

By important works \cite{leb-1, leb-2, leb-4, leb-11, leb-12, leb-17} of LeBrun, it is now well known that the Seiberg-Witten monopole equations \cite{w} lead to a remarkable family of curvature estimates which has many strong applications to 4-dimensional geometry. In this section, following a recent beautiful article \cite{leb-17} of LeBrun, we shall recall firstly these curvature estimates in terms of the {\it convex hull} of the set of all monopole classes on 4-manifolds. We shall use these estimates to prove new obstructions to the existence of non-singular solutions to the normalized Ricci flow in Section 5 below. The main results of this section are Theorems \ref{yamabe-pere} and \ref{key-mono-b} below. \par 
For the convenience of the reader who is unfamiliar with Seiberg-Witten theory, we shall recall briefy  the definition of the Seiberg-Witten monopole equations. Let $X$ be a closed oriented Riemannian 4-manifold and we assume that $X$ satisfies ${b}^{+}(X) \geq 2$, where ${b}^{+}(X)$ stands again for the dimension of a maximal positive definite subspace of ${H}^{2}(X, {\mathbb R})$ with respect to the intersection form. Recall that a ${spin}^{c}$-structure $\Gamma_{X}$ on a smooth Riemannian 4-manifold $X$ induces a pair of spinor bundles ${S}^{\pm}_{\Gamma_{X}}$ which are Hermitian vector bundles of rank 2 over $X$. A Riemannian metric on $X$ and a unitary connection $A$ on the determinant line bundle ${\cal L}_{\Gamma_{X}} := det({S}^{+}_{\Gamma_{X}})$ induce the twisted Dirac operator ${\cal D}_{{A}} : \Gamma({S}^{+}_{\Gamma_{X}}) \longrightarrow \Gamma({S}^{-}_{\Gamma_{X}})$. The Seiberg-Witten monopole equations over $X$ are the following system of non-linear partial differential equations for a unitary connection $A \in {\cal A}_{{\cal L}_{\Gamma_{X}}}$ and a spinor $\phi \in \Gamma({S}^{+}_{\Gamma_{X}})$:
\begin{eqnarray}\label{sw-mono}
{\cal D}_{{A}}{\phi} = 0, \ {F}^{+}_{{A}} = iq(\phi),  
\end{eqnarray}
here ${F}^{+}_{{A}}$ is the self-dual part of the curvature of $A$ and $q : {S}^{+}_{\Gamma_{X}} \rightarrow {\wedge}^{+}$ is a certain natural real-quadratic map satifying 
\begin{eqnarray*}
|q(\phi)|=\frac{1}{2\sqrt{2}}|\phi|^2, 
\end{eqnarray*}
where ${\wedge}^{+}$ is the bundle of self-dual 2-forms. \par
We are now in a position to recall the definition of monopole class \cite{kro, leb-11, ishi-leb-1, ishi-leb-2, leb-17}.
\begin{defn}\label{ishi-leb-2-key}
Let $X$ be a closed oriented smooth 4-manifold with $b^+(X) \geq 2$. An element $\frak{a} \in H^2(X, {\mathbb Z})$/torsion $\subset H^2(X, {\mathbb R})$ is called monopole class of $X$ if there exists a spin${}^c$ structure $\Gamma_{X}$ with 
\begin{eqnarray*}
{c}^{\mathbb R}_{1}({\cal L}_{\Gamma_{X}}) = \frak{a} 
\end{eqnarray*} 
which has the property that the corresponding Seiberg-Witten monopole equations (\ref{sw-mono}) have a solution for every Riemannian metric on $X$. Here ${c}^{\mathbb R}_{1}({\cal L}_{\Gamma_{X}})$ is the image of the first Chern class ${c}_{1}({\cal L}_{\Gamma_{X}})$ of the complex line bundle ${\cal L}_{\Gamma_{X}}$ in $H^2(X, {\mathbb R})$. We shall denote the set of all monopole classes on $X$ by ${\frak C}(X)$.
\end{defn} 

Crucial properties of the set ${\frak C}(X)$ are summarized as follow \cite{leb-17, ishi-leb-2}:

\begin{prop}[\cite{leb-17}]\label{mono}
Let $X$ be a closed oriented smooth 4-manifold with $b^+(X) \geq 2$. Then ${\frak C}(X)$ is a finite set. Morever ${\frak C}(X) = -{\frak C}(X)$ holds, i.e., $\frak{a} \in H^2(X, {\mathbb R})$ is a monopole class if and only if $-\frak{a} \in H^2(X, {\mathbb R})$ is a monopole class, too. 
\end{prop}

These properties of ${\frak C}(X)$ which sits in a real vector space $H^2(X, {\mathbb R})$ natually lead us to consider the convex hull ${\bf{Hull}}({\frak C}(X))$ of ${\frak C}(X)$. Recall that, for any subset $W$ of a real vector space $V$, one can consider the convex hull ${\bf{Hull}}(W) \subset V$, meaning the smallest convex subset of $V$ containg $W$. Then, Proposition \ref{mono} immediately implies the following result:
\begin{prop}[\cite{leb-17}]\label{mono-leb}
Let $X$ be a closed oriented smooth 4-manifold with $b^+(X) \geq 2$. Then the convex hull ${\bf{Hull}}({\frak C}(X)) \subset H^2(X, {\mathbb R})$ of ${\frak C}(X)$ is  compact, and symmetric, i.e., ${\bf{Hull}}({\frak C}(X)) = -{\bf{Hull}}({\frak C}(X))$. 
\end{prop}

By Proposition \ref{mono}, ${\frak C}(X)$ is a finite set and hence we are able to write as ${\frak C}(X)=\{{\frak{a}}_{1},{\frak{a}}_{2}, \cdots, {\frak{a}}_{n} \}$. The convex hull ${\bf{Hull}}({\frak C}(X))$ is then expressed as follows:\begin{eqnarray}\label{hull}
{\bf{Hull}}({\frak C}(X))= \{ \sum^{n}_{i=1}t_{i} {\frak{a}}_{i} \ | \ t_{i} \in [0,1], \ \sum^{n}_{i=1}t_{i}=1 \}. 
\end{eqnarray}
Notice also that the symmetric property tells us that ${\bf{Hull}}({\frak C}(X))$ contains the zero element. \par 
Now, consider the following self-intersection function:
\begin{eqnarray*}
{\cal Q} : H^2(X, {\mathbb R}) \rightarrow {\mathbb R}
\end{eqnarray*}
which is defined by $x \mapsto x^2:=<x \cup x, [X]>$, where $[X]$ is the fundamental class of $X$. Since this function ${\cal Q}$ is a polynomial function and hence is a continuous function on $H^2(X, {\mathbb R})$. We can therefore conclude that the restriction ${\cal Q} |_{{\bf{Hull}}({\frak C}(X))}$ to the compact subset ${\bf{Hull}}({\frak C}(X))$ of $H^2(X, {\mathbb R})$ must achieve its maximum. This leads us naturally to introduce the following quantity ${\beta}^2(X)$:
\begin{defn}[\cite{leb-17}]\label{beta}
Suppose that $X$ is a closed oriented smooth 4-manifold with $b^+(X) \geq 2$. Let ${\bf{Hull}}({\frak C}(X)) \subset H^2(X, {\mathbb R})$ be the convex hull of the set ${\frak C}(X)$ of all monopole classes on $X$. If ${\frak C}(X) \not= \emptyset$, define
\begin{eqnarray*}
{\beta}^2(X):= \max \{ {\cal Q}(x):=x^2 \ | \ x \in {\bf{Hull}}({\frak C}(X)) \}.
\end{eqnarray*}
On the other hand, if ${\frak C}(X) = \emptyset$ holds, define simply as ${\beta}^2(X):=0$.
\end{defn}
Since ${\bf{Hull}}({\frak C}(X))$ contains the zero element, the above definition particularly implies that ${\beta}^2(X) \geq 0$ holds. \par
On the other hand, the Hodge star operator associated to a given metric $g$ defines an involution on the real vector space $H^2(X, {\mathbb R})$ and this gives rise to an eigenspace decomposition:
\begin{eqnarray}\label{h-de}
H^2(X, {\mathbb R}) = {\cal H}^+_{g} \oplus {\cal H}^-_{g}, 
\end{eqnarray}
where ${\cal H}^{\pm}_{g}:=\{ \psi \in \Gamma({\wedge}^{\pm}) \ | \ d \psi=0 \}$ are the space of self-dual and anti-self-dual harmonic 2-forms. Notice that this decomposition depends on the metric $g$. This dependence also can be described in terms of the {\it period map}. In fact, consider the following map which is so called the period map of the Riemannian 4-manifold $X$: 
\begin{eqnarray}\label{h-de-p}
{\cal P} : {\cal R}_{X} \longrightarrow {Gr}^+_{b^+(X)} \Big(H^2(X, {\mathbb R}) \Big)
\end{eqnarray}
which is defined by $g \mapsto {\cal H}^+_{g}$. Here, ${\cal R}_{X}$ is the infinite dimensional space of all Riemannian metrics on $X$ and ${Gr}^+_{b^+} \Big(H^2(X, {\mathbb R}) \Big)$ is the finite dimensional Grassmannian of $b^+(X)$-dimensional subspace of $H^2(X, {\mathbb R})$ on which the intersection form of $X$ is positive definite. Namely, we are able to conclude that the decomposition (\ref{h-de}) depends on the image of the metric $g$ under the period map (\ref{h-de-p}). \par
Now, let $\frak{a} \in H^2(X, {\mathbb R})$ be a monopole class of $X$. Then we can consider the self-dual part $\frak{a}^+$ of $\frak{a}$ with respect to the decompsition (\ref{h-de}) and take square $\big(\frak{a}^+ \big)^2$. From the above argument, it is clear that this quantity $\big(\frak{a}^+ \big)^2$ also depends on the image of the meric $g$ under the period map (\ref{h-de-p}). On the other hand, the quantity ${\beta}^2(X)$ introduced in Definition \ref{beta} above dose not depend on the metric and hence it never depend on the period map (\ref{h-de-p}). One of important observations made in \cite{leb-17} is the following:
\begin{prop}[\cite{leb-17}]\label{main-leb}
Let $X$ be a closed oriented smooth 4-manifold with $b^+(X) \geq 2$. Suppose that ${\frak C}(X) \not= \emptyset$. Then, for any Riemannian metric $g$ on $X$, there is a monopole class $\frak{a} \in {\frak C}(X)$ satisfying 
\begin{eqnarray}\label{h-de-p-1}
\big(\frak{a}^+ \big)^2 \geq {\beta}^2(X). 
\end{eqnarray}
\end{prop}

On the other hand, it is well known that, as was firstly pointed out by Witten \cite{w}, the existence of a monopole class gives rise to a priori lower bound on the $L^2$-norm of the scalar curvature of Riemannian metrics. Its refined version is proved by LeBrun \cite{leb-2, leb-17}:
\begin{prop}[\cite{leb-2, leb-17}]\label{slight-leb}
Let $X$ be a closed oriented smooth 4-manifold with $b^+(X) \geq 2$ and a monopole class $\frak{a} \in H^2(X, {\mathbb Z})$/torsion $\subset H^2(X, {\mathbb R})$. Let $g$ be any Riemannian metric on $X$ and let $\frak{a}^+$ be the self-dual part of $\frak{a}$ with respect to the decomposition $H^2(X, {\mathbb R}) = {\cal H}^+_{g} \oplus {\cal H}^-_{g}$, identified with the space of $g$-harmonic 2-forms, into eigenspaces of the Hodge star operator. Then, the scalar curvature $s_{g}$ of $g$ must satisfy the following bound:
\begin{eqnarray}\label{sca-leb}
{\int}_{X}{{s}^2_{g}}d{\mu}_{g} \geq {32}{\pi}^{2}\big(\frak{a}^+ \big)^2. 
\end{eqnarray}
If $\frak{a}^+ \not=0$, furthermore, equality holds if and only if there is an integrable complex structure $J$ with ${c}^{\mathbb R}_{1}(X, J)=\frak{a}$ such that $(X, g, J)$ is a K{\"{a}}hler manifold of constant negative scalar curvature. 
\end{prop}

In \cite{leb-11, leb-17}, LeBrun moreover finds that the existence of a monopole class implies an estimate involving both the scalar curvature and the self-dual Weyl curvature as follows: 
\begin{prop}[\cite{leb-11, leb-17}]\label{fami-leb}
Let $X$ be a closed oriented smooth 4-manifold with $b^+(X) \geq 2$ and a monopole class $\frak{a} \in H^2(X, {\mathbb Z})$/torsion $\subset H^2(X, {\mathbb R})$. Let $g$ be any Riemannian metric on $X$ and let $\frak{a}^+$ be the self dual part of $\frak{a}$ with respect to the decomposition $H^2(X, {\mathbb R}) = {\cal H}^+_{g} \oplus {\cal H}^-_{g}$, identified with the space of $g$-harmonic 2-forms, into eigenspaces of the Hodge star operator. Then, the scalar curvature $s_{g}$ and the self-dual Weyl curvature $W^{+}_{g}$ of $g$ satisfy the following:
\begin{eqnarray}\label{weyl-leb}
{\int}_{X}\Big({s}_{g}-\sqrt{6}|W^{+}_{g}|\Big)^2 d{\mu}_{g} \geq 72{\pi}^{2}\big(\frak{a}^+ \big)^2, 
\end{eqnarray}
where the point wise norm are calculated with respect to $g$. And if $\frak{a}^+ \not =0$, furthermore, equality holds if and only if there is a symplectic form $\omega$, where the deRham class $[\omega]$ is negative multiple of $\frak{a}^+$ and ${c}^{\mathbb R}_{1}(X, \omega)=\frak{a}$, such that $(X, g, \omega)$ is a almost complex-K{\"{a}}hler manifold with the following complicated properties: 
\begin{itemize}
\item $2{s}_{g} + |\nabla \omega|^2$ is a negative constant;
\item $\omega$ belongs to the lowest eigenspace of $W^+_{g} : \wedge^+ \rightarrow \wedge^+$ everywhere; and 
\item the two largest eigenvalues of $W^+_{g} : \wedge^+ \rightarrow \wedge^+$ are everywhere equal. 
\end{itemize} 
\end{prop} 

Notice that, as was already mentioned above, the lower bounds of both (\ref{sca-leb}) and (\ref{weyl-leb}) depend on the image of the Riemmanin metric under the period map (\ref{h-de-p}). This means that these curvature estimates are not uniform in the metric. Propositions \ref{main-leb}, \ref{slight-leb} and \ref{fami-leb} together imply, however, the following curvature estimates which do not depend on the image of the Riemannian metric under the period map (\ref{h-de-p}):  
\begin{thm}[\cite{leb-17}]\label{beta-ine-key}
Suppose that $X$ is a closed oriented smooth 4-manifold with $b^+(X) \geq 2$. Then any Riemannian metric $g$ on $X$ satisfies the following curvature estimates:
\begin{eqnarray}\label{weyl-leb-sca-1}
{\int}_{X}{{s}^2_{g}}d{\mu}_{g} \geq {32}{\pi}^{2}\beta^2(X), 
\end{eqnarray}
\begin{eqnarray}\label{weyl-leb-sca-2}
{\int}_{X}\Big({s}_{g}-\sqrt{6}|W^{+}_{g}|\Big)^2 d{\mu}_{g} \geq 72{\pi}^{2}\beta^2(X), 
\end{eqnarray}
where $s_{g}$ and $W^{+}_{g}$ denote respectively the scalar curvature and the self-dual Weyl curvature of $g$. If $X$ has a non-zero monopole class and, moreover, equality occurs in either the first or the second estimate if and only if $g$ is a K{\"{a}}hler-Einstein metric with negative scalar curvature. 
\end{thm}
Notice that if $X$ has no monopole class, we define as $\beta^2(X):=0$ (see Definition \ref{beta} above). On the other hand, notice also that the left-hand side of these two curvature estimates in Theorem \ref{beta-ine-key} is always non-negative. Therefore, Propositions \ref{main-leb}, \ref{slight-leb} and \ref{fami-leb} indeed tell us that the desired estimates hold. To prove the statement of the boundary case, we need to analyze the curvature estimates more deeply. See the proof of Theorem 4.10 in \cite{leb-17}. \par
As a corollary of the second curvature estimate, we particularly obtain the following curvature bound (cf. Proposition 3.1 in \cite{leb-11}):
\begin{cor}\label{bound-cor}
Let $X$ be a closed oriented smooth 4-manifold with $b^+(X) \geq 2$. Then any Riemannian metric $g$ on $X$ satisfies the following curvature estimate:
\begin{eqnarray}\label{monopole-123}
\frac{1}{4{\pi}^2}{\int}_{X}\Big(2|W^{+}_{g}|^2+\frac{{s}^2_{g}}{24}\Big) d{\mu}_{g} \geq  \frac{2}{3}\beta^2(X), 
\end{eqnarray}
where $s_{g}$ and $W^{+}_{g}$ denote respectively the scalar curvature and the self-dual Weyl curvature of $g$. If $X$ has a non-zero monopole class and, moreover, equality occurs in the above estimate if and and only if $g$ is a K{\"{a}}hler-Einstein metric with negative scalar curvature. 
\end{cor} 
\begin{proof}
First of all, we have the curvature estimate (\ref{weyl-leb-sca-2}): 
\begin{eqnarray}\label{dot-1}
{\int}_{X}\Big({s}_{g}-\sqrt{6}|W^{+}_{g}|\Big)^2 d{\mu}_{g} \geq 72{\pi}^{2}\beta^2(X). 
\end{eqnarray}
By multiplying this by $4/9$, we are able to get
\begin{eqnarray*}
{\int}_{X}\Big(\frac{2}{3}{s}_{g}-2\sqrt{\frac{2}{3}}|W^{+}_{g}|\Big)^2 d{\mu}_{g} \geq 32{\pi}^{2}\beta^2(X). 
\end{eqnarray*}
We are able to rewrite this estimate as follows:
\begin{eqnarray*}
||\frac{2}{3}{s}_{g}-2\sqrt{\frac{2}{3}}|W^{+}_{g}| || \geq 4\sqrt{2}{\pi}\sqrt{\beta^2(X)}, 
\end{eqnarray*}
where $|| \cdot ||$ is the $L^2$ norm with respect to $g$ and notice that we always have $\beta^2(X) \geq 0$. The rest of the proof is essentially the same with that of Proposition 3.1 in \cite{leb-11}. For completeness, let us include the proof. Indeed, by the triangle inequality, we get the following estimate from the above
\begin{eqnarray}\label{dot}
\frac{2}{3}||{s}_{g}|| +\frac{1}{3}||\sqrt{24} |W^{+}_{g}| || \geq 4\sqrt{2}{\pi}\sqrt{\beta^2(X)}. 
\end{eqnarray}
The left-hand side of this can be interpreted as the dot product in $\mathbb R^2$:
\begin{eqnarray*}
\Big(\frac{2}{3}, \frac{1}{3\sqrt{2}} \Big) \cdot \Big(||{s}_{g}||, || \sqrt{48}|W^{+}_{g}| || \Big)=\frac{2}{3}||{s}_{g}||+\frac{1}{3}||\sqrt{24} |W^{+}_{g}| ||. 
\end{eqnarray*}
By applying Cauchy-Schwartz inequality, we have
\begin{eqnarray}\label{Cauchy-Schwartz}
\Big( \Big(\frac{2}{3}\Big)^2 + \Big(\frac{1}{3\sqrt{2}} \Big)^2 \Big)^{\frac{1}{2}}\Big( {\int}_{X}({s}^2_{g}+48|W^{+}_{g}|^2)d{\mu}_{g} \Big)^{\frac{1}{2}} \geq \frac{2}{3}||{s}_{g}||+\frac{1}{3}||\sqrt{24} |W^{+}_{g}| ||.
\end{eqnarray}
On the other hand, notice that
\begin{eqnarray*}
\Big( \Big(\frac{2}{3}\Big)^2 + \Big(\frac{1}{3\sqrt{2}} \Big)^2 \Big)^{\frac{1}{2}}\Big( {\int}_{X}({s}^2_{g}+48|W^{+}_{g}|^2)d{\mu}_{g} \Big)^{\frac{1}{2}} = \frac{1}{\sqrt{2}}\Big({\int}_{X}({s}^2_{g}+48|W^{+}_{g}|^2)d{\mu}_{g}\Big)^{\frac{1}{2}}. 
\end{eqnarray*}
This with the bounds (\ref{dot}) and (\ref{Cauchy-Schwartz}) tells us that 
\begin{eqnarray*}
\frac{1}{\sqrt{2}}\Big({\int}_{X}({s}^2_{g}+48|W^{+}_{g}|^2)d{\mu}_{g}\Big)^{\frac{1}{2}} \geq 4\sqrt{2}{\pi}\sqrt{\beta^2(X)}. 
\end{eqnarray*}
Thus we have
\begin{eqnarray*}
\frac{1}{{2}}{\int}_{X}\Big({s}^2_{g}+48|W^{+}_{g}|^2 \Big)d{\mu}_{g} \geq 32{\pi}^2\beta^2(X). 
\end{eqnarray*}
This immediately implies the desired bound:
\begin{eqnarray*}
\frac{1}{4{\pi}^2}{\int}_{X}\Big(2|W^{+}_{g}|^2+\frac{{s}^2_{g}}{24}\Big) d{\mu}_{g} \geq  \frac{2}{3}\beta^2(X).
\end{eqnarray*}
Finally, if $X$ has a non-zero monopole class and, moreover, equality occurs in the above estimate, then the above argument tells us that equality occurs in (\ref{dot-1}). Therefore the last claim follows from the last assertion in Theorem \ref{beta-ine-key}. 
\end{proof} 

On the other hand, we use the following result to prove Theorem \ref{yamabe-pere} below: 
\begin{prop}[\cite{leb-17}]\label{positive-mono}
Let $X$ be a closed oriented smooth 4-manifold with $b^+(X) \geq 2$. If there is a non-zero monopole class $\frak{a} \in {H}^2(X, {\mathbb R}) - \{0\}$, then $X$ cannot admit any Riemannian metric $g$ of scalar curvature $s_{g} \geq 0$.

\end{prop}
This result is well known to experts in Seiberg-Witten theory. We would like to notice that, however, a complete proof appears firstly in the proof of Proposition 3.3 in \cite{leb-17}. \par
On the other hand, there are several ways to detect the existence of monopole classes. For example, if $X$ is a closed symplectic 4-manifold $X$ with ${b}^{+}(X) \geq 2$, then $\pm c_{1}(K_{X})$ are both monopole classes by the celebrated result of Taubes \cite{t-1}, where $c_{1}({X})$ is the first Chern class of the canonical bundle of $X$. This is proved by thinking the moduli space of solutions of the Seiberg-Witten monopole equations as a cycle which represents an element of the homology of a certain configuration space. More precisely, for any closed oriented smooth 4-manifold $X$ with $b^+(X) \geq 2$, one can define the integer valued Seiberg-Witten invariant $SW_{X}(\Gamma_{X}) \in {\mathbb Z}$ for any spin${}^{c}$-structure $\Gamma_{X}$ by integrating a cohomology class on the moduli space of solutions of the Seiberg-Witten monopole equations associated to $\Gamma_{X}$:
\begin{eqnarray*}
SW_{X} : Spin(X) \longrightarrow {\mathbb Z}, 
\end{eqnarray*}
where $Spin(X)$ is the set of all spin${}^c$-structures on $X$. For more details, see \cite{w, nico}. Taubes indeed proved that, for any closed symplectic 4-manifold $X$ with ${b}^{+}(X) \geq 2$, $SW_{X}(\hat{\Gamma}_{X}) \equiv 1 \ (\bmod \ 2)$ holds for the canonical spin${}^{c}$-structure $\hat{\Gamma}_{X}$ induced from the symplectic structure. This actually implies that $\pm c_{1}(K_{X})$ are monopole classes of $X$. \par
On the other hand, there is a sophisticated refinement of the idea of this construction. It detects the presence of a monopole class by element of a stable cohomotopy group. This is due to Bauer and Furuta \cite{b-f, b-1}. They interpreted Seiberg-Witten monopole equations as a map between two Hilbert bundles over the Picard tours of a 4-manifold $X$. The map is called the Seiberg-Witten map (or monopole map). Roughly speaking, the cohomotopy refinement of the integer valued Seiberg-Witten invariant is defined by taking an equivariant stable cohomotopy class of the finite dimensional approximation of the Seiberg-Witten map. The invariant takes its value in a certain complicated equivariant stable cohomotopy group. We notice that Seiberg-Witten moduli space {\it does not} appear in their story. By using the stable cohomotopy refinement of Seiberg-Witten invariant, the following result is proved essentially by LeBrun with the present author (Proposition 10 and Corollary 11 in \cite{ishi-leb-2}):
\begin{prop}\label{prop-2}
For $i= 1,2,3,4$, suppose that $X_{i}$ is a closed almost-complex 4-manifold whose integer valued Seiberg-Witten invariant satisfies $SW_{X_{i}}(\Gamma_{X_{i}}) \equiv 1 \ (\bmod \ 2)$, where $\Gamma_{X_{i}}$ is the spin${}^c$-structure compatible with the almost-complex structure. Moreover assume that the following conditions are satisfied:
\begin{itemize}
\item $b_{1}(X_{i})=0$, \ $b^{+}(X_{i}) \equiv 3 \ (\bmod \ 4)$, 
\item $\displaystyle\sum^{4}_{i=1}b^{+}(X_{i}) \equiv 4 \ (\bmod \ 8)$. 
\end{itemize} 
Suppose that $N$ is a closed oriented smooth 4-manifold with $b^{+}(N)=0$ and let $E_{1}, E_{2}, \cdots, E_{k}$ be a set of generators for $H^2(N, {\mathbb Z})$/torsion relative to which the intersection form is diagonal. Then, for any $j=1, 2,3,4$, 
\begin{eqnarray}\label{mono-cone}
\sum^{j}_{i=1} \pm {c}_{1}(X_{i}) + \sum^{k}_{i=1} \pm{E}_{i} 
\end{eqnarray}
is a monopole class of $M:=\Big(\#^{j}_{i=1}{X}_{i} \Big) \# N$, where ${c}_{1}(X_{i})$ is the first Chern class of the canonical bundle of the almost-complex 4-manifold $X_{i}$ and the $\pm$ signs are arbitrary, and are independent of one another. Moreover, for any $j=1, 2,3,4$, 
\begin{eqnarray}\label{monopole-123446}
\beta^2(M) \geq  \sum^{j}_{i=1}{c}^2_{1}(X_{i}). 
\end{eqnarray}
\end{prop} 
\begin{proof}
Thanks to Proposition 10 in \cite{ishi-leb-2}, it is enough to prove the bound (\ref{monopole-123446}) only. First of all, by the very definition, we have 
\begin{eqnarray*}
{\beta}^2(M):= \max \{ {\cal Q}(x):=x^2 \ | \ x \in {\bf{Hull}}({\frak C}(M)) \}.
\end{eqnarray*}
On the other hand, by (\ref{mono-cone}), we especially have the following two monopole classes of $M$:
\begin{eqnarray*}
{\frak{a}}_{1}:=\sum^{j}_{i=1} {c}_{1}(X_{i}) + \sum^{k}_{i=1} {E}_{i}, \ {\frak{a}}_{2}:=\sum^{j}_{i=1} {c}_{1}(X_{i}) - \sum^{k}_{i=1} {E}_{i}. 
\end{eqnarray*}
By (\ref{hull}), we are able to conclude that
\begin{eqnarray*}
\sum^{j}_{i=1} {c}_{1}(X_{i})= \frac{1}{2}{\frak{a}}_{1}+\frac{1}{2}{\frak{a}}_{2} \in {\bf{Hull}}({\frak C}(M)). 
\end{eqnarray*}
We therefore obtain 
\begin{eqnarray*}
{\beta}^2(M) \geq \Big( \sum^{j}_{i=1} {c}_{1}(X_{i})\Big)^2=\sum^{j}_{i=1}{c}^2_{1}(X_{i})
\end{eqnarray*}
as desired. 
\end{proof}

Notice here that, in case of $j=1$, we assume that $b_{1}=0$ and $b^{+} \equiv 3 \ (\bmod \ 4)$ hold. It turns out that, however, these conditions are superfluous though such a thing is not asserted in \cite{ishi-leb-2}. In fact, we are able to show 
\begin{prop}\label{prop-3}
Let $X$ be a closed almost-complex 4-manifold with a non-trivial integer valued Seiberg-Witten invariant $SW_{X}(\Gamma_{X}) \not=0$, where $\Gamma_{X}$ is the spin${}^c$-structure compatible with the almost complex structure. Let $N$ be a closed oriented smooth 4-manifold with $b^{+}(N)=0$ and let $E_{1}, E_{2}, \cdots, E_{k}$ be a set of generators for $H^2(N, {\mathbb Z})$/torsion relative to which the intersection form is diagonal. Then, 
\begin{eqnarray*}
\pm {c}_{1}(X) + \sum^{k}_{i=1} \pm{E}_{i} 
\end{eqnarray*}
is a monopole class of $M:={X} \# N$, where ${c}_{1}(X)$ is the first Chern class of the canonical bundle of $X$ and the $\pm$ signs are arbitrary, and are independent of one another. Moreover, the following holds:
\begin{eqnarray}\label{monopole-1234}
\beta^2(M) \geq {c}^2_{1}(X). 
\end{eqnarray}
\end{prop}
\begin{proof}
It is known that there is a comparision map between the stable cohomotopy refinement of Seiberg-Witten invariant and the integer valued Seiberg-Witten invariant \cite{b-f, b-2}. In particular, Proposition 5.4 in \cite{b-2} tells us that the comparision map becomes isomorphism when the given 4-manifold is almost-complex and ${b}^+ > 1$. Hence, the value of Bauer-Furuta's stable cohomotopy invariant of $X$ for the spin${}^c$-structure $\Gamma_{X}$ compatible with the almost complex structure is non-trivial if $X$ is a closed almost-complex 4-manifold with a non-trivial integer valued Seiberg-Witten invariant $SW_{X}(\Gamma_{X}) \not=0$. Moreover, the proofs of Proposition 6 and Corollary 8 in \cite{ishi-leb-2} (see also Theorem 8.8 in \cite{b-2}) imply that
\begin{eqnarray}\label{mono-cone-1}
\pm {c}_{1}(X) + \sum^{k}_{i=1} \pm{E}_{i} 
\end{eqnarray}
is indeed a monopole class of the connected sum $M:={X} \# N$. \par
On the other hand, the last claim follows as follows. Indeed, by (\ref{mono-cone-1}), we are able to obtain the following two monopole classes of $M$:
\begin{eqnarray*}
{\frak{b}}_{1}:={c}_{1}(X) + \sum^{k}_{i=1} {E}_{i}, \ {\frak{b}}_{2}:={c}_{1}(X) - \sum^{k}_{i=1} {E}_{i}. 
\end{eqnarray*}
By (\ref{hull}), we obtain
\begin{eqnarray*}
{c}_{1}(X)= \frac{1}{2}{\frak{b}}_{1}+\frac{1}{2}{\frak{b}}_{2} \in {\bf{Hull}}({\frak C}(M)). 
\end{eqnarray*}
We therefore get
\begin{eqnarray*}
{\beta}^2(M) \geq {c}^2_{1}(X) 
\end{eqnarray*}
as promised. 
\end{proof}

Theorem \ref{AIL}, Theorem \ref{beta-ine-key}, Proposition \ref{positive-mono}, Proposition \ref{prop-2}, and Proposition \ref{prop-3} together imply Theorem \ref{yamabe-pere} below. We shall use Theorem \ref{yamabe-pere} in next section. Moreover, Theorem \ref{yamabe-pere} is of interest independently of the applications to the Ricci flow. Compare Theorem \ref{yamabe-pere} with Theorem A in \cite{ishi-leb-2} and several results of \cite{fang}: 
\begin{thm}\label{yamabe-pere}
Let $N$ be a closed oriented smooth 4-manifold with $b^{+}(N)=0$. Let $X$ be a closed almost-complex 4-manifold with $b^{+}(X) \geq 2$ and $c^2_{1}(X)=2\chi(X) + 3 \tau(X) > 0$ . Assume that $X$ has a non-trivial integer valued Seiberg-Witten invariant $SW_{X}(\Gamma_{X}) \not=0$, where $\Gamma_{X}$ is the spin${}^c$-structure compatible with the almost-complex structure. Then, 
\begin{eqnarray}\label{one-ya}
{\mathcal Y}(X \# N) = \bar{\lambda}(X \# N) \leq -4{\pi}\sqrt{2c^2_{1}(X)} < 0. 
\end{eqnarray}
Moreover, if $X$ a minimal K{\"{a}}hler surface and if $N$ admits a Riemannian metric of non-negative scalar curvature, then,
\begin{eqnarray*}
{\mathcal Y}(X \# N) = \bar{\lambda}(X \# N) = -4{\pi}\sqrt{2c^2_{1}(X)} < 0. 
\end{eqnarray*}
On the other hand, let ${X}_{i}$ be as in Proposition \ref{prop-2} and assume that $\sum^j_{i=1}c^2_{1}(X_{i})=\sum^j_{i=1}(2\chi(X_{i}) + 3 \tau(X_{i})) > 0$ is satisfied, where $j=2,3,4$. For $j=2,3,4$, 
\begin{eqnarray}\label{se-ya}
{\mathcal Y}((\#^{j}_{i=1}{X}_{i}) \# N) = \bar{\lambda}((\#^{j}_{i=1}{X}_{i}) \# N) \leq -4{\pi}\sqrt{2\sum^j_{i=1}c^2_{1}(X_{i})} < 0. 
\end{eqnarray}
Moreover, if $X_{i}$ is a minimal K{\"{a}}hler surface, where $i=1,2,3,4$, and if $N$ admits a Riemannian metric of non-negative scalar curvature, then,
\begin{eqnarray*}
{\mathcal Y}((\#^{j}_{i=1}{X}_{i}) \# N) = \bar{\lambda}((\#^{j}_{i=1}{X}_{i}) \# N) = -4{\pi}\sqrt{2\sum^j_{i=1}c^2_{1}(X_{i})} < 0. 
\end{eqnarray*}
\end{thm}
\begin{proof}
First of all, the condition that $c^2_{1}(X)=2\chi(X) + 3 \tau(X) > 0$ forces that $\frak{a}:=c^{\mathbb R}_{1}({\cal L}_{\Gamma_{X}})$ is a non-zero monopole class. This fact with Proposition \ref{prop-3} allows us to conclude that the connected sum $X \# N$ has non-zero monopole classes. By Proposition \ref{positive-mono} and this fact, $X \# N$ does not admit any Riemannian metric $g$ with $s_{g} \geq 0$. This particularly implies that the Yamabe invariant of $X \# N$ is non-positive. By formula (\ref{yama-def-2}), we are able to obtain
\begin{eqnarray}\label{sca-yama}
{\mathcal Y}(X \# N) = - \Big(\inf_{g}{\int}_{X\# N}s^{{2}}_{g} d{\mu}_{g} \Big)^{{1}/{2}}. 
\end{eqnarray}
On the other hand, the bounds (\ref{weyl-leb-sca-1}) and (\ref{monopole-1234}) immediately imply
\begin{eqnarray}\label{sca-1}
{\mathcal I}_{s}(X \# N):=\inf_{g}{\int}_{X \# N}{{s}^2_{g}}d{\mu}_{g} \geq {32}{\pi}^{2}{c}^2_{1}(X).
\end{eqnarray}
We are therefore able to obtain the desired bound (\ref{one-ya}), where we used Theorem \ref{AIL}. On the ther hand, it is known that, for any minimal K{\"{a}}hler surface $X$ with $b^+(X) \geq 2$, ${\mathcal I}_{s}(X)=32{\pi}^2c^2_{1}(X)$ holds \cite{leb-4, leb-7}. Moreover, ${\mathcal I}_{s}(N)=0$ holds because we assume that $N$ admits a Riemannian metric of non-negative scalar curvature. Proposition 13 of \cite{ishi-leb-2} with these facts together tells us that 
\begin{eqnarray}\label{sca-2}
{\mathcal I}_{s}(X \# N) \leq {\mathcal I}_{s}(X) + {\mathcal I}_{s}(N) = 32{\pi}^2c^2_{1}(X).
\end{eqnarray}
It is clear that (\ref{sca-1}) and (\ref{sca-2}) imply ${\mathcal I}_{s}(X \# N)= {32}{\pi}^{2}{c}^2_{1}(X)$. This equality with (\ref{sca-yama}) and Theorem \ref{AIL} gives us the desired equality:
\begin{eqnarray*}
{\mathcal Y}(X \# N) = \bar{\lambda}(X \# N) = -4{\pi}\sqrt{2c^2_{1}(X)}. 
\end{eqnarray*}
We should notice that, in case where $b_{1}(X)=0$ and $b^{+}(X) \equiv 3 \ (\bmod \ 4)$, this result can be recovered from Theorem \ref{AIL} and Theorem A of \cite{ishi-leb-2}. Moreover, the bound (\ref{se-ya}) is also essentially proved in \cite{ishi-leb-2}. For the reader, we shall include a proof. The method is quite similar to the above. In fact, since we know that $(\#^{j}_{i=1}{X}_{i}) \# N$ has non-zero monopole classes, the bounds (\ref{weyl-leb-sca-1}) and (\ref{monopole-123446}) tell us that the following holds for $j=2,3,4$:
\begin{eqnarray*}
{\mathcal I}_{s}((\#^{j}_{i=1}{X}_{i}) \# N):=\inf_{g}{\int}_{(\#^{j}_{i=1}{X}_{i}) \# N}{{s}^2_{g}}d{\mu}_{g} \geq {32}{\pi}^{2}\sum^{j}_{i=1}{c}^2_{1}(X_{i}).
\end{eqnarray*}
This bound with Theorem \ref{AIL} implies the desired bound (\ref{se-ya}) because the existence of non-zero monopole classes of $(\#^{j}_{i=1}{X}_{i}) \# N$ forces that \begin{eqnarray}\label{sp}
{\mathcal Y}((\#^{j}_{i=1}{X}_{i}) \# N) = - \Big(\inf_{g}{\int}_{(\#^{j}_{i=1}{X}_{i}) \# N}s^{{2}}_{g} d{\mu}_{g} \Big)^{{1}/{2}}
\end{eqnarray} 
as before. On the other hand, if $X_{i}$ is a minimal K{\"{a}}hler surface, here $i=1,2,3,4$, then Proposition 13 in \cite{ishi-leb-2} tells us that
\begin{eqnarray*}
{\mathcal I}_{s}((\#^{j}_{i=1}{X}_{i}) \# N) \leq \sum^{j}_{i=1}{\mathcal I}_{s}(X_{i}) + {\mathcal I}_{s}(N) = 32{\pi}^2\sum^{j}_{i=1}{c}^2_{1}(X_{i}), 
\end{eqnarray*}
where we again used ${\mathcal I}_{s}(N)=0$. We therefore get ${\mathcal I}_{s}((\#^{j}_{i=1}{X}_{i}) \# N) = 32{\pi}^2\sum^{j}_{i=1}{c}^2_{1}(X_{i})$. This equality with (\ref{sp}) and Theorem \ref{AIL} implies 
\begin{eqnarray*}
{\mathcal Y}((\#^{j}_{i=1}{X}_{i}) \# N) = \bar{\lambda}((\#^{j}_{i=1}{X}_{i}) \# N) = -4{\pi}\sqrt{2\sum^j_{i=1}c^2_{1}(X_{i})}. 
\end{eqnarray*}
Hence we obtain the promised result. 
\end{proof} 

As was already mentioned in Introduction, it is known that the Yamabe invariant is sensitive to the choice of smooth structure of a 4-manifold. In fact, one can easily construct many examples of compact topological 4-manifolds admitting distinct smooth structures for which values of the Yamabe invariants are different by using Theorem \ref{yamabe-pere}. We leave it as an excercise for the interested reader. \par
We shall close this section with the following result. The bounds (\ref{monopole-123}), (\ref{monopole-123446}) and (\ref{monopole-1234}) immedialtely imply the following important result for our purpose:
\begin{thm}\label{key-mono-b}
Let $N$ be a closed oriented smooth 4-manifold with $b^{+}(N)=0$. Let $X$ be a closed almost-complex 4-manifold with ${b}^+(X) \geq 2$ and with a non-trivial integer valued Seiberg-Witten invariant $SW_{X_{i}}(\Gamma_{X_{i}}) \not=0$, where $\Gamma_{X}$ is the spin${}^c$-structure compatible with the almost-complex structure. Then, any Riemannian metric $g$ on the connected sum $M_{1}:={X} \# N$ satisfies 
\begin{eqnarray}\label{monopolee-1}
\frac{1}{4{\pi}^2}{\int}_{M_{1}}\Big(2|W^{+}_{g}|^2+\frac{{s}^2_{g}}{24}\Big) d{\mu}_{g} \geq \frac{2}{3} c^2_{1}({X}). 
\end{eqnarray}
On the other hand, let ${X}_{i}$ be as in Proposition \ref{prop-2}. For $j=2,3,4$, any Riemannian metric $g$ on the connected sum $M_{2}:=\Big(\#^{j}_{i=1}{X}_{i} \Big) \# N$ satisfies the following strict bound:
\begin{eqnarray}\label{monopoleee-2}
\frac{1}{4{\pi}^2}{\int}_{M_{2}}\Big(2|W^{+}_{g}|^2+\frac{{s}^2_{g}}{24}\Big) d{\mu}_{g} \geq \frac{2}{3} \sum^{j}_{i=1}c^2_{1}({X}_{i}). 
\end{eqnarray}
\end{thm}

\section{Obstructions to the Existence of Non-Singular Solutions to the Normalized Ricci Flow}\label{obstruction}

In this section, we shall prove new obstructions to the existence of non-singular solutions to the normalized Ricci flow by using results proved in the previous several sections. One of the main results of this section is the following: 

\begin{thm}\label{ricci-ob-1}
Let $N$ be a closed oriented smooth 4-manifold with $b^{+}(N)=0$. Let $X$ be a closed almost-complex 4-manifold with ${b}^+(X) \geq 2$ and ${c}^2_{1}(X)=2\chi(X) + 3 \tau(X)>0$. Assume that $X$ has a non-trivial integer valued Seiberg-Witten invariant $SW_{X}(\Gamma_{X}) \not=0$, where $\Gamma_{X}$ is the spin${}^c$-structure compatible with the almost-complex structure. Then, there does not exist quasi-non-singular solutions to the normalized Ricci flow in the sense of Definition \ref{bs} on a connected sum $M:=X \# N$ if the following holds:
\begin{eqnarray}\label{ob-N-Ricci}
(12b_{1}(N) + 3{b}^{-}(N)) > {c}^2_{1}(X). 
\end{eqnarray}
In particular, under this condition, there does not exist non-singular solutions to the normalized Ricci flow in the sense of Definition \ref{non-sin}. 
\end{thm}

\begin{proof}
Suppose that there is a quasi-non-singular solution $\{g(t)\}$, $t \in [0, \infty)$, to the normalized Ricci flow on $M:=X \# N$. First of all, the bound (\ref{one-ya}) in Theorem \ref{yamabe-pere} tells us that
\begin{eqnarray*}
{\mathcal Y}(M) = \bar{\lambda}(M) \leq -4{\pi}\sqrt{2c^2_{1}(X)} < 0, 
\end{eqnarray*}
where notice that the assumption that ${c}^2_{1}(X)=2\chi(X) + 3 \tau(X)>0$. Theorem \ref{bound-four} therefore tells us that the connected sum $M$ must satisfy the strict FZZ inequality. More precisely, as was already seen in the proof of Theorem \ref{bound-four} or Theorem \ref{fz-key}, the following holds:
\begin{eqnarray*}
2\chi(M) + 3\tau(M) \geq  \liminf_{m \longrightarrow \infty}\frac{1}{4{\pi}^2}{\int}^{m+1}_{m} {\int}_{M}\Big(2|W^{+}_{g(t)}|^2+\frac{{s}^2_{g(t)}}{24}\Big) d{\mu}_{g(t)}dt. 
\end{eqnarray*}
On the other hand, by the bound (\ref{monopolee-1}) in Theorem \ref{key-mono-b}, we get the following bound for any Riemannian metric $g$ on $M$:
\begin{eqnarray*}
\frac{1}{4{\pi}^2}{\int}_{M}\Big(2|W^{+}_{g}|^2+\frac{{s}^2_{g}}{24}\Big) d{\mu}_{g} \geq \frac{2}{3} c^2_{1}({X}). 
\end{eqnarray*}
We therefore obtain 
\begin{eqnarray*}
2\chi(M) + 3\tau(M) &\geq&  \liminf_{m \longrightarrow \infty}\frac{1}{4{\pi}^2}{\int}^{m+1}_{m} {\int}_{M}\Big(2|W^{+}_{g(t)}|^2+\frac{{s}^2_{g(t)}}{24}\Big) d{\mu}_{g(t)}dt \\
&\geq&\ \frac{2}{3}c^2_{1}({X}). 
\end{eqnarray*}
On the other hand, a direct computation tells us that
\begin{eqnarray*}
2\chi(M) + 3\tau(M) &=& 2\chi(X) + 3 \tau(X) + (2\chi(N) + 3 \tau(N)) -4 \\
&=& c^2_{1}({X}) - (4b_{1}(N) + {b}^{-}(N)), 
\end{eqnarray*}
where we used the assumption that ${b}^{+}(N)=0$. We therefore obtain 
\begin{eqnarray*}
c^2_{1}({X}) - (4b_{1}(N) + {b}^{-}(N)) \geq \frac{2}{3}c^2_{1}({X}) . 
\end{eqnarray*}
Namely, 
\begin{eqnarray*}
(12b_{1}(N) + 3{b}^{-}(N)) \leq {c}^2_{1}(X). 
\end{eqnarray*}
By contraposition, we are able to get the desired result. 
\end{proof}

In Section \ref{final-main} below, we shall actually use the following special case of Theorem \ref{ricci-ob-1}, but, a slightly stronger result in a sense:
\begin{cor}\label{non-sin-cor}
Let $X$ be a closed symplectic 4-manifold with $b^{+}(X) \geq 2$ and ${c}^2_{1}(X) >0$. Then, there is no non-singular solution of the normalized Ricci flow on a connected sum $M:=X \# k{\overline{{\mathbb C}{P}^2}}$ if the following holds:
\begin{eqnarray}\label{Ricci-sym}
k \geq \frac{1}{3}{c}^2_{1}(X). 
\end{eqnarray}
\end{cor}
\begin{proof}
Let us again recall that a celebrated result of Taubes \cite{t-1} asserts that, for any symplectic 4-manifold with $b^+(X)>1$, the integer valued Seiberg-Witten invariant satisfies $SW_{X}(\Gamma_{X}) \equiv 1 \ (\bmod \ 2)$, where $\Gamma_{X}$ is the canonical spin${}^c$ structure compatible with the symplectic structure. Notice also that $k{\overline{{\mathbb C}{P}^2}}$ satisfies $b^+=0$. These facts with (\ref{ob-N-Ricci}) tell us that, if $3k > {c}^2_{1}(X)$, then there is no non-singular solution of the normalized Ricci flow on $M$. However, notice that the symplectic 4-manifold $M$ cannot admit any K{\"{a}}hler-Einstein metric with negative scalar curvature if $k > 0$. This particularly implies the following strict bound:
\begin{eqnarray*}
\frac{1}{4{\pi}^2}{\int}_{M}\Big(2|W^{+}_{g}|^2+\frac{{s}^2_{g}}{24}\Big) d{\mu}_{g} > \frac{2}{3} c^2_{1}({X}),  
\end{eqnarray*}
here see also Corollary \ref{bound-cor}. This bound and the above proof of Theorem \ref{ricci-ob-1} immediately implies the slightly strong bound (\ref{Ricci-sym}) as desired. 
\end{proof} 

A similar method also allows us to prove the following obstruction which is the second main result of this section:
\begin{thm}\label{ricci-ob-2}
For $i= 1,2,3,4$, let $X_{i}$ be a closed almost-complex 4-manifold whose integer valued Seiberg-Witten invariant satisfies $SW_{X_{i}}(\Gamma_{X_{i}}) \equiv 1 \ (\bmod \ 2)$, where $\Gamma_{X_{i}}$ is the spin${}^c$-structure compactible with the almost-complex structure. Assume that the following conditions are satisfied:
\begin{itemize}
\item $b_{1}(X_{i})=0$, \ $b^{+}(X_{i}) \equiv 3 \ (\bmod \ 4)$, \ $\displaystyle\sum^{4}_{i=1}b^{+}(X_{i}) \equiv 4 \ (\bmod \ 8)$, 
\item $\displaystyle\sum^j_{i=1}c^2_{1}(X_{i})=\sum^j_{i=1}(2\chi(X_{i}) + 3 \tau(X_{i})) > 0$, where $j=2,3,4$. 
\end{itemize} 
Let $N$ be a closed oriented smooth 4-manifold with $b^{+}(N)=0$. Then, for $j=2,3,4$, there does not exist quasi-non-singular solutions to the normalized Ricci flow in the sense of Definition \ref{bs} on a connected sum $M:=\Big(\#^{j}_{i=1}{X}_{i} \Big) \# N$ if the following holds:
\begin{eqnarray*}
12(j-1)+(12b_{1}(N) + 3{b}^{-}(N)) \geq  \sum^{j}_{i=1}{c}^2_{1}(X_{i}) . 
\end{eqnarray*}
In particular, under this condition, there does not exist non-singular solutions to the normalized Ricci flow on $M$ in the sense of Definition \ref{non-sin}. 
\end{thm} 
\begin{proof}
Suppose now that there is a quasi-non-singular solution $\{g(t)\}$, $t \in [0, \infty)$, to the normalized Ricci flow on $M$. The bound (\ref{se-ya}) in Theorem \ref{yamabe-pere} tells us that
\begin{eqnarray*}
{\mathcal Y}(M) = \bar{\lambda}(M) \leq -4{\pi}\sqrt{2\sum^{j}_{i=1}c^2_{1}(X_{i})} < 0. 
\end{eqnarray*}
This particularly tells us that, as before, the following must hold (see the proofs of Theorem \ref{fz-key} and Theorem \ref{bound-four} above)
\begin{eqnarray*}
2\chi(M) + 3\tau(M) \geq  \liminf_{m \longrightarrow \infty}\frac{1}{4{\pi}^2}{\int}^{m+1}_{m} {\int}_{M}\Big(2|W^{+}_{g(t)}|^2+\frac{{s}^2_{g(t)}}{24}\Big) d{\mu}_{g(t)}dt.  
\end{eqnarray*}
On the other hand, notice that the connected sum $M$ admits non-zero monopole classes and cannot admit symplectic structures. This fact and Theorem \ref{beta-ine-key} tell us that the bound (\ref{monopoleee-2}) must be strict: 
\begin{eqnarray*}
\frac{1}{4{\pi}^2}{\int}_{M}\Big(2|W^{+}_{g(t)}|^2+\frac{{s}^2_{g(t)}}{24}\Big) d{\mu}_{g(t)} > \frac{2}{3}\sum^{j}_{i=1}{c}^2_{1}(X_{i}). 
\end{eqnarray*}
We therefore obtain
\begin{eqnarray*}
2\chi(M) + 3\tau(M) &\geq&  \liminf_{m \longrightarrow \infty}\frac{1}{4{\pi}^2}{\int}^{m+1}_{m} {\int}_{M}\Big(2|W^{+}_{g(t)}|^2+\frac{{s}^2_{g(t)}}{24}\Big) d{\mu}_{g(t)}dt \\
&>&\ \frac{2}{3}\sum^{j}_{i=1}{c}^2_{1}(X_{i}). 
\end{eqnarray*}
On the other hand, a direct computation implies 
\begin{eqnarray*}
2\chi(M) + 3\tau(M) &=& \sum^{j}_{i=1}(2\chi(X_{i}) + 3 \tau(X_{i})) + (2\chi(N) + 3 \tau(N)) -4j \\
&=& -(4b_{1}(N) + {b}^{-}(N)) - 4(j-1) + \sum^{j}_{i=1}{c}^2_{1}(X_{i}) , 
\end{eqnarray*}
where we used the assumption that ${b}^{+}(N)=0$. We therefore get 
\begin{eqnarray*}
-(4b_{1}(N) + {b}^{-}(N)) - 4(j-1)+\sum^{j}_{i=1}{c}^2_{1}(X_{i})  > \frac{2}{3}\sum^{j}_{i=1}{c}^2_{1}(X_{i}). 
\end{eqnarray*}
Namely, we have 
\begin{eqnarray*}
12(j-1)+(12b_{1}(N) + 3{b}^{-}(N)) < \sum^{j}_{i=1}{c}^2_{1}(X_{i}). 
\end{eqnarray*}
By contraposition, the desired result follows. 
\end{proof}  

Theorem \ref{ricci-ob-2}, a result of Taubes \cite{t-1} and the fact that a connected sum $k{\overline{{\mathbb C}{P}^2}} \# {\ell}({S^1} \times {S}^3)$ satisfies $b^+=0$ enable us to prove
\begin{cor}\label{main-cor}
For $i=1,2,3,4$, let ${X}_{i}$ be a simply connected closed symplectic 4-manifold satifying
\begin{itemize}
\item $b^{+}(X_{i}) \equiv 3 \ (\bmod \ 4)$, \ $\displaystyle\sum^{4}_{i=1}b^{+}(X_{i}) \equiv 4 \ (\bmod \ 8)$, 
\item $\displaystyle\sum^j_{i=1}c^2_{1}(X_{i})=\sum^j_{i=1}(2\chi(X_{i}) + 3 \tau(X_{i})) > 0$, where $j=2,3,4$. 
\end{itemize} 
Then, for $j=2,3,4$, there is also no non-singular solution to the normalized Ricci flow on a connected sum $\Big(\#^{j}_{i=1}{X}_{i} \Big) \# k{\overline{{\mathbb C}{P}^2}} \# {\ell}({S^1} \times {S}^3)$ if the following holds:
\begin{eqnarray*}\label{Ricci-sym-1}
12(j-1)+12{\ell}+3k \geq  \sum^{j}_{i=1}{c}^2_{1}(X_{i}). 
\end{eqnarray*}
Similarly, for $j=2,3,4$, there is also no non-singular solution to the normalized Ricci flow on $\#^{j}_{i=1}{X}_{i}$ if the following holds:
\begin{eqnarray*}\label{ishi-leb-ob-22222}
12(j-1) \geq \sum^{j}_{i=1}{c}^2_{1}(X_{i}). 
\end{eqnarray*}
\end{cor}

Let us close this section with the following result. Though it is not used in what follows, perhaps, it is worth pointing out that the following holds (cf. Corollary 1.5 in \cite{fz-1}, Theorems 5.1 and 5.2 in \cite{leb-17}):
\begin{thm}\label{ricci-ein-ob}
Let $X$ be a closed oriented smooth 4-manifold with $b^+(X) \geq 2$. Suppose that there is a quasi-non-singular solution $\{g(t)\}$, $t \in [0, \infty)$, to the normalized Ricci flow in the sense of Definition \ref{bs}. If the Yamabe invariant of $X$ is negative, i.e., ${\mathcal Y}(X)<0$, then the following two inequalities hold:
\begin{eqnarray}\label{mi-ya-ricci-2}
2 \chi(X) +3\tau(X) \geq \frac{2}{3}\beta^2(X), 
\end{eqnarray}
\begin{eqnarray}\label{mi-ya-ricci}
2 \chi(X) -3\tau(X) \geq \frac{1}{3}\beta^2(X). 
\end{eqnarray}
In particular, if $X$ is a closed almost-complex 4-manifold with a non-trivial integer valued Seiberg-Witten invariant $SW_{X}(\Gamma_{X}) \not=0$, where $\Gamma_{X}$ is the spin${}^c$-structure compatible with the almost-complex structure, then the bound (\ref{mi-ya-ricci}) implies the Bogomolov-Miyaoka-Yau type inequality:
\begin{eqnarray*}
\chi(X) \geq 3\tau(X). 
\end{eqnarray*}
\end{thm}
\begin{proof}
By the assumption that ${\mathcal Y}(X)<0$ and the proof of Theorem \ref{bound-four} above, we know that the existence of quasi-non-singular solution $\{g(t)\}$, $t \in [0, \infty)$, to the normalized Ricci flow implies the following: 
\begin{eqnarray*}
2\chi(M) \pm 3\tau(M) \geq  \liminf_{m \longrightarrow \infty}\frac{1}{4{\pi}^2}{\int}^{m+1}_{m} {\int}_{M}\Big(2|W^{\pm}_{g(t)}|^2+\frac{{s}^2_{g(t)}}{24}\Big) d{\mu}_{g(t)}dt.  
\end{eqnarray*}
The inequality (\ref{mi-ya-ricci}) is derived from this and (\ref{weyl-leb-sca-1}). In fact, 
\begin{eqnarray*}
2\chi(X) - 3\tau(X) &\geq&  \liminf_{m \longrightarrow \infty}\frac{1}{4{\pi}^2}{\int}^{m+1}_{m} {\int}_{X}\Big(2|W^{-}_{g(t)}|^2+\frac{{s}^2_{g(t)}}{24}\Big) d{\mu}_{g(t)}dt \\
 &\geq& \liminf_{m \longrightarrow \infty}\frac{1}{96{\pi}^2}{\int}^{m+1}_{m} {\int}_{X}{{s}^2_{g(t)}} d{\mu}_{g(t)}dt \\
 &\geq& \frac{1}{3}\beta^2(X).  
\end{eqnarray*}
We used (\ref{weyl-leb-sca-1}) in the last part. Moreover, suppose that $X$ is a closed almost-complex 4-manifold with a non-trivial integer valued Seiberg-Witten invariant $SW_{X}(\Gamma_{X}) \not=0$. Then, the bound (\ref{monopole-1234})  particularly tells us that the following holds:
\begin{eqnarray}
\beta^2(X) \geq {c}^2_{1}(X)= 2\chi(X) + 3\tau(X). 
\end{eqnarray}
We therefore get 
$$ 
2\chi(X) - 3\tau(X) \geq \frac{1}{3}(2\chi(X) + 3\tau(X)). 
$$
Namely, we obtain 
\begin{eqnarray*}
\chi(X) \geq 3\tau(X)
\end{eqnarray*}
as promised. \par
Finally, we also have 
\begin{eqnarray*}
2\chi(M) + 3\tau(M) \geq \liminf_{m \longrightarrow \infty}\frac{1}{4{\pi}^2}{\int}^{m+1}_{m} {\int}_{M}\Big(2|W^{+}_{g(t)}|^2+\frac{{s}^2_{g(t)}}{24}\Big) d{\mu}_{g(t)}dt. 
\end{eqnarray*}
This bound with (\ref{monopole-123}) immediately implies the desired inequality: 
\begin{eqnarray*}
2\chi(M) + 3\tau(M) \geq  \frac{2}{3}\beta^2(X).  
\end{eqnarray*}
Hence the claim follows. 
\end{proof}

\begin{rmk}
Both (\ref{mi-ya-ricci-2}) and (\ref{mi-ya-ricci}) still hold even if $\beta^2(X)$ is replaced by $\alpha^2(X)$ which is introduced in \cite{leb-12, leb-17}. For the reader, let us recall briefly the defintion of $\alpha^2(X)$. Let $X$ be a closed oriented smooth 4-manifold with $b^+(X) \geq 2$. Consider the Grassmannian ${\bf Gr}:={Gr}^+_{b^+} \Big(H^2(X, {\mathbb R}) \Big)$ which consists of all maximal linear subspaces $\bf H$ of $H^2(X, {\mathbb R})$ on which the intersection form of $X$ is positive definite. For each element ${\bf H} \in {\bf Gr}$, we have an orthogonal decomposition with respect to the intersection form:
$$
H^2(X, {\mathbb R}) = {\bf H} \oplus \overline{\bf H}. 
$$
Hence, for a given monopole class $\frak{a} \in {\frak C}(X)$ and an element ${\bf H} \in {\bf Gr}$, one can define $\frak{a}^+$ to be the orthogonal projection of $\frak{a}$ to ${\bf H}$. Using this projection, we can define the following natural quantity:
$$
\alpha^2(X) := \inf_{{\bf H} \in {\bf Gr}} \Big(\max_{\frak{a}\in {\frak C}(X)}(\frak{a}^+)^2 \Big). 
$$
Though this definition is totally different from that of $\beta^2(X)$, it is observed in \cite{leb-17} that $\alpha^2(X) = \beta^2(X)$ actually occurs in many cases. In this direction, see Section 5 of \cite{leb-17}. 
\end{rmk}

\section{Proof of Theorem \ref{main-A}}\label{final-main}

In this section, we shall give a proof of Theorem \ref{main-A}. In what follows, we shall use the following notation:
\begin{eqnarray*}
{\chi}_{h}(X):=\frac{1}{4}\Big(\chi(X) + \tau(X)\Big), \ {c}^{2}_{1}(X):=2\chi(X) + 3\tau(X)
\end{eqnarray*}
for any 4-manifold $X$. \par
First of all, we shall prove the following result by using the obstruction proved in Corollary \ref{non-sin-cor} above:
\begin{prop}\label{non-prop}
For every $\delta>0$, there exists a constant $d_\delta>0$ satisfying the following property: every lattice point $(\alpha, \beta)$ satisfying
\begin{eqnarray}\label{geo}
0 < \beta \leq (6-\delta)\alpha-d_\delta
\end{eqnarray}
is realized by $({\chi}_{h}, {c}^{2}_{1})$ of infinitely many pairwise non-diffeomorphic simply connected symplectic 4-manifolds with the following properties:
\begin{itemize}
\item each symplectic 4-manifold $N$ is non-spin, 
\item each symplectic 4-manifold $N$ has negative Yamabe and Pelerman's $\bar{\lambda}$ invariant, i.e., ${\mathcal Y}(N)=\bar{\lambda}(N) <0$, 
\item on each symplectic 4-manifold $N$, there exists no quasi-non-singular solution of the normalized Ricci flow in the sense of Definition \ref{bs}. In particular, there is also no non-singular solution of the normalized Ricci flow in the sense of Definition \ref{non-sin}. 
\end{itemize}

\end{prop}
\begin{proof}
Building upon symplectic sum construction due to Gompf \cite{g} and gluing formula of Seinerg-Witten invariants due to Morgan-Mrowka-Szab{\'{o}} \cite{mms} and Morgan-Szab{\'{o}}-Taubes \cite{mst}, a nice result on infinitely many pairwise non-diffeomorphic simply connected symplectic 4-manifolds is proved in \cite{b-k}. In particualr, infinitely many smooth structures are given by performing the logarithmic transformation in the sense of Kodaira. Theorem 4 of \cite{b-k} tells us that, for every $\delta>0$, there exists a constant $d_\delta>0$ satisfying the following property: every lattice point $(\alpha, \beta)$ satisfying
$$
0 < \beta \leq (9-\delta)\alpha-d_\delta
$$
is realiezed by $({\chi}_{h}, {c}^{2}_{1})$ of infinitely many pairwise non-diffeomorphic simply connected symplectic 4-manifolds. In particular, each symplectic 4-manifold $X$ satisfies ${c}^{2}_{1}(X)=\beta > 0$ and we are able to know that $b^+(X) \geq 2$ by the construction. By the bound (\ref{Ricci-sym}), we are able to conclude that, if a positive integer $k$ satisfies
\begin{eqnarray*}
k \geq \frac{1}{3}{c}^2_{1}(X) = \frac{\beta}{3}, 
\end{eqnarray*}
then there exists no quasi-non-singular solution to the normalized Ricci flow on the symplectic 4-manifold $N:=X \# k \overline{{\mathbb C}{P}^2}$. Moreover, $N:=X \# k \overline{{\mathbb C}{P}^2}$ is non-spin. These non-spin symplectic 4-manifolds actually cover the area (\ref{geo}) and here notice also that 
\begin{eqnarray*}
{\chi}_{h}(N) = {\chi}_{h}(X), \ c^{2}_{1}(N)=\beta-k. 
\end{eqnarray*}
Moreover, under the connected sum with $\overline{{\mathbb C}{P}^2}$, the infinitely many different smooth structures remain distinct as was already noticed in \cite{b-k}. Finally, since $X$ has non-trivial valued Seiberg-Witten invariant by a result of Taubes \cite{t-1}, the bound (\ref{one-ya}) tells us that
\begin{eqnarray*}
{\mathcal Y}(N) = \bar{\lambda}(N) \leq -4{\pi}\sqrt{2c^2_{1}(X)}=-4{\pi}\sqrt{2\beta}<0. 
\end{eqnarray*}
We therefore obtain the desired result. 
\end{proof}

\begin{rmk}
By using Corollary \ref{main-cor}, Proposition \ref{non-prop} and Theorem 4 of \cite{b-k}, it is not hard to prove the following general non-existence result on non-singular solution: for every $\delta >0$, there is a constant $d_{\delta}>0$ such that a non-spin 4-manifold $m{\mathbb C}{P}^2 \# n \overline{{\mathbb C}{P}^2}$ has infinitely many smooth structures with ${\mathcal Y}<0$ for which there exists no non-singular solution to the normalized Ricci flow for every large enough $m \not\equiv 0 \ (\bmod \ 8)$ and $n \geq (2+\delta)m + d_{\delta}$. The details are left to the interested reader. Under these conditions, the author does not know, however, whether or not $m{\mathbb C}{P}^2 \# n \overline{{\mathbb C}{P}^2}$ admits actually a smooth structure for which non-singular solutions of the normalized Ricci flow exist. 
\end{rmk}

On the other hand, there is a nice result of Cao \cite{c, c-c} concerning the existence of non-singular solutions to the normalized Ricci flow. We shall recall the following version of Cao's result which appears in \cite{c-c}. 
\begin{thm}[\cite{c, c-c}]\label{cao-K}
Let $M$ be a compact K{\"{a}}hler manifold with definite first Chern class ${c}_{1}(M)$. If ${c}_{1}(M)=0$, then for any initial K{\"{a}}hler metric $g_{0}$, the solution to the normalized Ricci flow exists for all time and converges to a Ricci-flat metric as $t \rightarrow \infty$. If ${c}_{1}(M) < 0$ and the initial metric $g_0$ is chosen to represent the first Chern class, then the solution to the normalized Ricci flow exists for all time and converges to an Einstein metric of negative scalar curvature as $t \rightarrow \infty$. If ${c}_{1}(M) > 0$ and the initial metric $g_0$ is chosen to represent the first Chern class, then the solution to the normalized Ricci flow exists for all time. 
\end{thm}
Notice that, in case where ${c}_{1}(M) = 0$ or ${c}_{1}(M) < 0$, the solution is actually non-singular in the sense of Definition \ref{non-sin}. Notice also that the affirmative answer of the Calabi conjecture due to Aubin \cite{a} and Yau \cite{yau, yau-1} tells us that K{\"{a}}hler-Einstein metrics exist in these cases. See also Section 4 in \cite{c-c}. We shall use Theorem \ref{cao-K} to prove
\begin{prop}\label{exis-prop}
For every positive integer $\ell > 0$, there are $\ell$-tuples of simply connected spin and non-spin algebraic surfaces with the following properties: 
\begin{itemize}
\item these are homeomorhic, but are pairwise non-diffeomorphic, 
\item for every fixed $\ell > 0$, the ratios $c^2_{1}/{\chi}_{h}$ of the $\ell$-tuples are dense in the interval $[4,8]$, 
\item each algebraic surface $M$ has negative Yamabe and Pelerman's $\bar{\lambda}$ invariant, i.e., ${\mathcal Y}(M)=\bar{\lambda}(M) <0$, 
\item on each algebraic surface $M$, there exists a non-singular solution to the the normalized Ricci flow in the sense of Definition \ref{non-sin}. Moreover the existence of the solution forces the strict FZZ inequality $2 \chi(M)> 3|\tau(M)|$ as a topological constraint.
\end{itemize}
\end{prop}

\begin{proof}
Salvetti \cite{sal} proved that, for any $k > 0$, there exists a pair $(\chi_h, c^{2}_{1})$ such that for this pair one has at least $k$ homeomorphic algebraic surfaces with different divisibilities for their canonical classes by taking iterated branched covers of the projective plane. This construction is fairly generalized in \cite{b-k}. By Corollary 1 of \cite{b-k}, we know that, for every $\ell$, there are $\ell$-tuples of simply connected spin and non-spin algebraic surfaces with ample canonical bundles which are homeomorphic, but are pairwise non-diffeomorphic. Moreover, it is shown that, for every fixed $\ell$, the ratios $c^2_{1}/\chi_{h}$ of the $\ell$-tuples are dense in the interval $[4,8]$. Therefore, to prove this proposition, it is enough to prove the third and fourth statements above. We notice that one can see that each such an algebraic surface $M$ has $b^+(M) \geq 3$ by the construction. Now, the negativity of the Yamabe and Pelerman's $\bar{\lambda}$ invar!
 iant of the algebraic surface $M$ is a direct consequence of Theorem \ref{yamabe-pere}. In fact, the canonical bundle of each algebraic surface $M$ is ample and hence ${c}_{1}(M) < 0$. In particular, since $M$ is a minimal K{\"{a}}hler surface with ${b}^{+}_{2}(M) \geq 3$ and ${c}^2_{1}(M) > 0$, Theorem \ref{yamabe-pere} tells us that
\begin{eqnarray*}
{\mathcal Y}(M) = \bar{\lambda}(M) = -4{\pi}\sqrt{2c^2_{1}(M)} < 0. 
\end{eqnarray*}
Hence the third statement follows. \par
The fourth statement follows from Theorem \ref{cao-K} above because each algebraic surface $M$ has ample canonical bundle and hence ${c}_{1}(M) < 0$. We therefore conclude that, for the initial metric $g_0$ which is chosen to represent the first Chern class, there always exists a non-singular solution to the normalized Ricci flow and it converges to an Einstein metric of negative scalar curvature as $t \rightarrow \infty$. On the other hand, notice that the non-singular solution is particularly a quasi-non-singular solution in the sense of Definition \ref{bs}. Theorem \ref{bound-four} and the fact that $M$ has negative Yamabe invariant imply that $M$ must satisfy the strict FZZ inequality $2 \chi(M)> 3|\tau(M)|$ as a topological constraint. 
\end{proof}

Propositions \ref{non-prop} and \ref{exis-prop} enable us to prove the main result of this article, i.e., Theorem \ref{main-A} stated in Introduction:

\begin{thm}
For every natural number $\ell$, there exist a simply connected topological non-spin 4-manifold $X_{\ell}$ satisfying the following properties:
\begin{itemize}
\item $X_{\ell}$ admits at least $\ell$ different smooth structures $M^i_{\ell}$ with ${\mathcal Y}<0$ and for which there exist non-singular solutions to the the normalized Ricci flow in the sense of Definition \ref{non-sin}. Moreover, the existence of the solutions forces the strict FZZ inequality $2 \chi > 3|\tau|$ as a topological constraint, 
\item $X_{\ell}$ also admits infinitely many different smooth structures  $N^j_{\ell}$ with ${\mathcal Y}<0$ and for which there exists no quasi-non-singular solution to the normalized Ricci flow in the sense of Definition \ref{bs}. In particular, there exists no non-singular solution to the the normalized Ricci flow in the sense of Definition \ref{non-sin}.
\end{itemize} 
\end{thm}

\begin{proof}
Proposition \ref{exis-prop} tells us that, for every positive integer $\ell > 0$, we are always able to find $\ell$-tuples $M^i_{\ell}$ of simply connected non-spin algebraic surfaces of general type and these are homeomorhic, but are pairwise non-diffeomorphic. And the ratios $c^2_{1}/\chi_h$ of $M^i_{\ell}$ are dense in the interval $[4,8]$ for every fixed $\ell > 0$. Moreover, Proposition \ref{exis-prop} tells us that each of $M^i_{\ell}$ has ${\mathcal Y}<0$ and, on each of $M^i_{\ell}$, there exists a non-singular solution to the the normalized Ricci flow and the existence of the solution forces the strict FZZ inequality $2 \chi> 3|\tau|$ as a topological constraint. \par
On the other hand, Proposition \ref{non-prop} tells us that any pair $(\alpha, \beta)$ in the area (\ref{geo}) can be realized by $(\chi_h, {c}^{2}_{1})$ of infinitely many pairwise non-diffeomorphic simply connected non-spin symplectic 4-manifolds with ${\mathcal Y}<0$ and on each of which there exists no quasi-non-singular solution of the normalized Ricci flow. Notice that the ratios $c^2_{1}/\chi_h$ of these non-spin symplectic 4-manifolds are not more than 6, here see again the area (\ref{geo}). By this fact and the density of the ratios $c^2_{1}/\chi_h$ of $M^i_{\ell}$ in the interval $[4,8]$, we are able to find infinitely many pairwise non-diffeomorphic simply connected non-spin symplectic 4-manifolds $N^i_{\ell}$ such that ${\mathcal Y}<0$ and, on each of $N^i_{\ell}$, there exists no quasi-non-singular solution of the normalized Ricci flow, and moreover, $M^i_{\ell}$ and $N^i_{\ell}$ are both non-spin and have the same $(\chi_h, {c}^{2}_{1})$. Freedman's classification \cite{free} implies that they must be homeomorphic. However, each of $M^i_{\ell}$ is not diffeomorphic to any $N^i_{\ell}$ because, on each of $M^i_{\ell}$, a non-singular solution exists and, on the other hand, no non-singular solution exists on each of $N^i_{\ell}$. Therefore, we are able to conclude that, for every natural number $\ell$, there exists a simply connected topological non-spin 4-manifold $X_{\ell}$ satisfying the desired properties. 
\end{proof}

\section{Concluding Remarks}\label{remark}

In this article, we have seen that the existence or non-existence of non-singular solutions to the normalized Ricci flow depends on the diffeotype of a 4-manifold and it is not determined by homeotype alone. In particular, we considered distinct smooth structures on simply connected topological non-spin 4-manifolds $p{\mathbb C}{P}^2 \# q \overline{{\mathbb C}{P}^2}$ in Theorem \ref{main-A}. Freedman's classification \cite{free} tells us that, up to homeomorphism, the connected sums $j{\mathbb C}{P}^2 \# k \overline{{\mathbb C}{P}^2}$ provides us with a complete list of the simply connected non-spin 4-manifolds. In light of this fact, it will be tempting to ask whether or not the phenomenon like Theorem \ref{main-A} is a general feature of the Ricci flow on simply connected non-spin 4-manifolds admitting exotic smooth structures. However, there are quite many difficulties to prove such a result and hence this is a completely open problem. \par
On the other hand, in case of topological spin 4-manifolds, the situation on homeotypes is a bit more unsettled. But, the connected sum $m(K3)\#n(S^2 \times S^2)$ and thier orientation-reversed version, together with 4-sphere $S^4$ at least exhaust all the simply connected homeotypes satifying $\chi \geq \frac{11}{8}|\tau|+2$. The $11/8$-conjecture asserts that this constraint is indeed satisfied automatically and hence that the above list of spin homeotypes is complete. Notice that there is a storng partial result due to Furuta \cite{f-1} which asserts that $\chi \geq \frac{10}{8}|\tau|+2$ holds. It will be also tempting to ask whether or not a result like Theorem \ref{main-A} still holds for the Ricci flow on simply connected spin 4-manifolds admitting exotic smooth structures. However, the present method cannot prove an abundance theorem like Theorem \ref{main-A} in spin case because the present method cannot prove a result like Proposition \ref{non-prop} in spin case. Hence the situation is quite different from the case of non-spin. \par
Finally, it will be also interesting to ask whether or not a phenomenon like Theorem \ref{main-A} still occurs in {\it non} simply connected case. We hope to return this interesting subject in further research. 


\vfill

{\footnotesize 
\noindent
{Masashi Ishida, \\
{Department of Mathematics,  
Sophia University, \\ 7-1 Kioi-Cho, Chiyoda-Ku, 
 Tokyo 102-8554, Japan }\\
{\sc e-mail}: masashi@math.sunysb.edu}

\end{document}